\documentclass[a4paper,11pt]{amsart}

\usepackage{amsmath,amssymb,amsthm}
\usepackage{graphicx,pxfonts}
\usepackage[all]{xy}

\DeclareMathAlphabet{\mathbbold}{U}{bbold}{m}{n}

\def\k{\mathbbold{k}}

\DeclareSymbolFont{rsfscript}{OMS}{rsfs}{m}{n}
\DeclareSymbolFontAlphabet{\mathrsfs}{rsfscript}

\DeclareFontFamily{OMS}{rsfs}{\skewchar\font'177}
\DeclareFontShape{OMS}{rsfs}{m}{n}{%
      <5> rsfs5
      <6> <7> rsfs7
      <8> <9> <10> rsfs10
      <10.95> <12> <14.4> <17.28> <20.74> <24.88> rsfs10
      }{}

\def\calA{\mathrsfs{A}}

\def\calD{\mathrsfs{D}}

\def\calF{\mathrsfs{F}}
\def\calG{\mathrsfs{G}}

\def\calM{\mathrsfs{M}}

\def\calO{\mathrsfs{O}}
\def\calP{\mathrsfs{P}}
\def\calQ{\mathrsfs{Q}}
\def\calR{\mathrsfs{R}}

\def\calV{\mathrsfs{V}}

\def\sfA{\mathsf{bv}}

\DeclareMathOperator{\id}{id}

\DeclareMathOperator{\lt}{lt}
\DeclareMathOperator{\As}{As}

\DeclareMathOperator{\AntiCom}{AntiCom}

\DeclareMathOperator{\End}{End}
\DeclareMathOperator{\Tor}{Tor}

\DeclareMathOperator{\Id}{Id}

\DeclareMathOperator{\Hycom}{Hycom}
\DeclareMathOperator{\Grav}{Grav}

\makeatletter

\theoremstyle{plain}

\newtheorem {theorem}{Theorem}
\newtheorem {lemma}{Lemma}[section]
\newtheorem {corollary}{Corollary}

\newtheorem {proposition}{Proposition}[section]

\theoremstyle{definition}

\newtheorem {definition}{Definition}
\newtheorem {remark}{Remark}
\newtheorem {example}{Example}

\newcommand{\ldot}{{\:\raisebox{1.5pt}{\selectfont\text{\circle*{1.5}}}}}
\newcommand{\udot}{{\:\raisebox{3pt}{\selectfont\text{\circle*{1.5}}}}}

\begin{document}

\title{Free resolutions via Gr\"obner bases}
\author{Vladimir Dotsenko}
\address{School of Mathematics, Trinity College, Dublin 2, Ireland}
\email{vdots@maths.tcd.ie}
\author{Anton Khoroshkin}
\address{Departement Matematik, ETH, R\"amistrasse 101, 8092 Zurich, Switzerland and 
ITEP, Bolshaya Cheremushkinskaya 25, 117259, Moscow, Russia}
\email{anton.khoroshkin@math.ethz.ch}

\thanks{The first author's research was supported by the grant RFBR-CNRS-07-01-92214 and by an IRCSET research fellowship. The second author's research was supported by grants
MK-4736.2008.1, NSh-3035.2008.2, RFBR-07-01-00526, RFBR 08-01-00110, RFBR-CNRS-07-01-92214, and by a ETH research fellowship.
}

\begin{abstract}
For associative algebras in many different categories, it is possible to develop the machinery of Gr\"obner bases. A Gr\"obner basis of defining relations for an algebra of such a category provides a ``monomial replacement'' of this algebra. The main goal of this article is to demonstrate how this machinery can be used for the purposes of homological algebra. More precisely, our approach goes in three steps. First, we define a combinatorial resolution for the monomial replacement of an object. Second, we extract from those resolutions explicit representatives for homological classes. Finally, we explain how to ``deform'' the differential to handle the general case. For associative algebras, we recover a well known construction due to Anick. The other case we discuss in detail is that of operads, where we discover resolutions that haven't been known previously. We present various applications, including a proofs of Hoffbeck's PBW criterion, a proof of Koszulness for a class of operads coming from commutative algebras, and a homology computation for the operads of Batalin--Vilkovisky algebras and of Rota--Baxter algebras.
\end{abstract}

\maketitle

\section{Introduction}

\subsection{Description of results}

For the purposes of homological and homotopical algebra, it is often important to have (quasi-)free resolutions for associative algebras (and their generalisations in various monoidal categories). One resolution readily available for a generic associative algebra is obtained by iterating the bar construction (to be more precise, by a cobar-bar construction), however, sometimes it is preferable (and possible) to have a much smaller resolution. The so called minimal resolution has the homology of the bar complex as its space of generators; that homology is also of independent interest because it describes higher syzygies of the given associative algebra (relations, relations between relations etc.). According to the general philosophy of homotopical algebra~\cite{QuillenComAlg,QuillenHomotopical}, homology of the differential induced on the space of indecomposable elements of a free resolution~$(\calF,d)\to A$ of the given associative algebra~$A$ does not depend on the choice of a resolution, and this homology coincides with the homology of the differential induced on the space of generators of a resolution of the trivial~$A$-module by free right $A$-modules. 

One of most important practical results provided (in many different frameworks) by Gr\"obner bases is that when dealing with various linear algebra information (bases, dimensions etc.) one can replace an algebra with complicated relations by an algebra with monomial relations without losing any information of that sort. When it comes to questions of homological algebra, things become more subtle, since (co)homology may ``jump up'' for a monomial replacement of an algebra. However, the idea of applying Gr\"obner bases to problems of homological algebra is far from hopeless. It turns out that for monomial algebras it is often possible to construct very neat resolutions that can be used for various computations; furthermore, the data computed by these resolutions can be used to obtain results in the general (not necessarily monomial) case. The mail goal of this paper is to explain this approach in detail for computations of the bar homology.

In the case of usual associative algebras, this approach has been known since the celebrated paper of Anick~\cite{Anick} where in the case of monomial relations a minimal right module resolution of the trivial module was computed, and an explicit way to deform the differential was presented to handle the general case. Later, the Anick's resolution was generalized to the case of categories by Malbos~\cite{Malbos} who also asked whether this work could be extended to the case of operads. In this paper, we answer that question, and propose a framework that allows to handle associative algebras presented via generators and relations in many different monoidal categories in a uniform way. Our approach goes as follows. We begin with a resolution which is sometimes larger than the one of Anick, but has the advantage of not using much information about the underlying monoidal category. It is based on the inclusion--exclusion principle, and is in a sense a version of the cluster method of enumerative combinatorics due to Goulden and Jackson~\cite{GJ}. Once the inclusion--exclusion resolution is constructed, we find explicit combinatorial formulas for homological classes of algebras with monomial relations. In the case of algebras, this immediately recovers ``chains'', as defined by Anick. This is followed, in the spirit of the Anick's approach, of how to adapt the differential of our resolution to incorporate lower terms of relations and handle arbitrary algebras with known Gr\"obner bases. 

Our main motivating example is the case of (symmetric) operads. In~\cite{DK}, we introduced new type of monoids based on nonsymmetric collections, shuffle operads, to develop the machinery of Gr\"obner bases for symmetric operads. Basically, there exists a monoidal structure on nonsymmetric collections (shuffle composition) for which the forgetful functor from symmetric collections is monoidal, and this reduces many computations in the symmetric category to those in the shuffle category. The shuffle category provides precisely what is needed to define operadic Gr\"obner bases, so our approach applies. Other categories where our methods are applied without much change are commutative associative algebras, associative dialgebras, (shuffle) coloured operads, dioperads, $\frac12$PROPs etc. We shall discuss details elsewhere. 

There are various applications of our approach; some of them are presented in this paper. Two interesting theoretical applications are a new short proof of Hoffbeck's PBW criterion for operads \cite{Hoffbeck}, and an upper bound on the homology for operads obtained from commutative algebras; in particular, we prove that an operad obtained from a Koszul commutative algebra is Koszul. Some interesting concrete examples where all steps of our construction can be completed are the case of the operad $RB$ of Rota--Baxter algebras, its noncommutative analogue~$ncRB$, and, the last but not the least, the operad~$BV$ of Batalin--Vilkovisky algebras. Using our methods, we were able to compute the bar homology of $BV$ and relate it to the gravity operad of Getzler~\cite{Getzler1}. While preparing our paper, we learned that these (and other) results on~$BV$ were announced earlier by Drummond-Cole and Vallette (see the extended abstract~\cite{DCVOberwolfach}, the slides~\cite{ValBV}, and the forthcoming paper~\cite{DCV}). Their approach relies, on one hand, on on theorems of Galvez-Carillo, Tonks and Vallette~\cite{GTV} who studied the operad $BV$ as an operad with nonhomogeneous (quadratic--linear) relations, and, on the other hand, on some new results in operadic homotopical algebra, in particular, a homotopy transfer theorem for infinity-cooperads. Our methods appear to be completely different: we treat~$BV$ as an operad with homogeneous relations of degrees $2$ and~$3$, and apply the Gr\"obner bases approach. We hope that our approach to the operad~$BV$ is also of independent interest as an illustration of a rather general method to compute the bar homology.

\subsection{Outline of the paper}
This paper is organised as follows. 
 
In Section~\ref{assoc_alg}, we handle (usual) associative algebras with monomial relations. We construct a free resolution for such algebras, and then propose a way to choose explicit representatives of homology classes. It turns out that we end up with the celebrated construction of Anick. We also explain how to ``deform'' the boundary maps in our resolutions for monomial algebras to obtain resolutions in the general case of a homogeneous Gr\"obner basis. Even though the corresponding homological results for algebras are very well known, we make a point of elaborating on our strategy since it turns out to be fruitful in a very general setting. We also discuss an interesting example of an algebra which exhibits most effects we shall later discover in the operad~$BV$. This algebra also appears to provide a counterexample to a theorem of Farkas~\cite{Farkas} on the structure of the Anick resolution.

In Section~\ref{operads}, we demonstrate how to apply our method to shuffle operads (primarily having in mind applications to symmetric operads, as mentioned above). We provide the readers with background information on tree monomials, the replacement for monomials in the case of free shuffle operads, and explain how to adapt all the results previously obtained in the case of algebras to the case of operads. Throughout this section, we apply our methods to a particular example, computing dimensions for the low arity homology of the operad of anti-associative algebras, and some low degree boundary maps in the corresponding resolution.
 
In Section~\ref{appl}, we exhibit applications of our results outlined above (a new proof of the PBW criterion, homology estimates for operads coming from commutative algebras, and a computation of the bar homology for the operads~$RB$, $ncRB$, and $BV$).

All vector spaces and (co)chain complexes throughout this work are defined over an arbitrary field~$\k$ of zero characteristic. 

\subsection{Acknowledgements}

We are thankful to Fr\'ed\'eric Chapoton, Iain Gordon, Ed Green, Eric Hoffbeck, Jean--Louis Loday, Dmitri Piontkovskii and Emil Sk{\"o}ldberg for useful discussions, remarks on a preliminary version of this paper and references to literature. Special thanks are due to Li Guo for drawing our attention to Rota--Baxter algebras as an example to which our methods could be applied, and especially to Bruno Vallette for many useful discussions and in particular for explaining the approach to the operad~$BV$ used in his joint work with Gabriel Drummond-Cole.

\section{Associative algebras}\label{assoc_alg}

\subsection{Inclusion-exclusion resolution for monomial algebras}\label{ass_monom}

In this section, we discuss the case of associative algebras with monomial relations. We start with an algebra $R=\k\langle x_1,\ldots,x_n\rangle/(g_1,\ldots,g_m)$ with $n$ generators and $m$ relations, each of which is a monomial in the given generators. We work under the assumption that none of the monomials $g_1,\ldots,g_m$ is divisible by another; this, for example, is the case when $G$ is the set of leading monomials of a reduced Gr\"obner basis of some algebra for which~$R$ is a monomial replacement. 

Let us denote by $A(p,q)$ the vector space whose basis is formed by elements of the form 
$x_I\otimes S_1\wedge S_2\wedge\ldots\wedge S_q$ where $I=(i_1,\ldots,i_p)\in[n]^p$,  $x_I=x_{i_1}\otimes\ldots\otimes x_{i_p}$ is the corresponding monomial, and $S_1,\ldots,S_q$ are (in one-to-one correspondence with) certain divisors $x_{i_r}x_{i_{r+1}}\ldots x_{i_s}$ of this monomial. Each $S_i$ is thought of as a  symbol of homological degree~$1$, with the appropriate Koszul sign rule for wedge products (here $q \ge 0$, so the wedge product might be empty). As we shall see later, in the classical approach for associative algebras there is no need in wedge products because there exists a natural linear ordering on all divisors of the given monomial. However, we introduce the wedge notation here as it becomes crucial for the general case (e.g. for operads). 

For each $p$ the graded vector space $A(p)=\bigoplus_q A(p,q)$ is a chain complex with the differential given by the usual formula with omitted factors:
\begin{equation}
d(x_I\otimes S_1\wedge \ldots\wedge S_q)= \sum_{l} (-1)^{l-1} x_I\otimes S_1\wedge\ldots\wedge\hat{S_l}\wedge\ldots\wedge S_q.  
\end{equation}

Moreover, there exists a natural algebra structure on~$A=\bigoplus_{p,q}A(p,q)$:
\begin{multline}
(x_I\otimes S_1\wedge \ldots\wedge S_q)\cdot
(x_{I'}\otimes T_1\wedge \ldots\wedge T_{q'}) \\
= (x_Ix_{I'})\otimes \imath(S_1)\wedge \ldots\wedge \imath(S_q)\wedge \imath(T_1)\wedge \ldots\wedge \imath(T_{q'}),
\end{multline}
where $\imath$ identifies divisors of $x_I$ and $x_{I'}$ with the corresponding divisors of $x_Ix_{I'}$.
The differential~$d$ makes~$A$ into an associative dg-algebra. 

It is important to emphasize that the symbols $S_i$ correspond to divisors, i.e. occurrences of monomials in $x_I$ rather than monomials themselves, so in particular the Koszul sign rule does not imply that  elements of our algebra square to zero. The following example should make our construction more clear.

\begin{example}
Assume that our algebra $R$ has just one generator~$x$. Then the algebra $A$ has an element $x\otimes S$ where $S$ corresponds to the divisor of the monomial $x$ equal to $x$ itself. We have
 $$
(x\otimes S)(x\otimes S)=x^2\otimes S_1\wedge S_2, 
 $$
where $S_1$ and $S_2$ indicate the two different divisors of $x^2$ equal to~$x$ (and the ``naive'' result of multiplication giving $x^2\otimes S\wedge S=0$ does not make sense at all, because $S$ corresponds to a divisor of $x$, not of $x^2$). Basically, when computing products, the $S$-symbols ``remember'' which divisors of factors they come from.
\end{example}

So far we did not use the relations of our algebra. Let us incorporate relations in the picture. We denote by $G=\{g_1,\ldots,g_m\}$ the set of relations of our algebra, and by $A_G$ the subspace of $A$ spanned by all elements $x_I\otimes S_1\wedge\ldots\wedge S_q$ for which the divisor corresponding to $S_j$ coincides, for every~$j$, with one of the relations from~$G$. This subspace is stable under product and differential, i.e. is a dg-subalgebra of~$A$. 

\begin{example}
Let us consider one of the simplest cases of an associative algebra, that of dual numbers: $R=k[x]/(x^2)$. Here $m=n=1$, $g_1=x^2$. Let us list all monomials in the corresponding algebra~$A$. Such a monomial is of the form $x^n\otimes \mathbf{S}$, where $\mathbf{S}$ is a wedge product of symbols corresponding to some of the divisors of $x^n$ equal to $x^2$. There are $n-1$ such divisors, the one covering the first two letters~$x$, the one covering the second and the third one etc. We denote those divisors by $S^{(n)}_1$, $S^{(n)}_2$,\ldots, $S^{(n)}_{n-1}$. Thus, $\mathbf{S}=S^{(n)}_{i_1}\wedge S^{(n)}_{i_2}\wedge\ldots\wedge S^{(n)}_{i_r}$. For instance, for $n=1$ the only monomial allowed is $x\otimes 1$, for $n=2$ there are two monomials, $x^2\otimes 1$ and $x^2\otimes S^{(2)}_1$, for $n=3$~--- four monomials $x^3\otimes1$, $x^3\otimes S^{(3)}_1$, $x^3\otimes S^{(3)}_2$, $x^3\otimes S^{(3)}_1\wedge S^{(3)}_2$. Products of those monomials are computed in a straightforward way, keeping in mind that the $S$-symbols control the location of corresponding relation in the underlying $x$-monomial. For instance,
\begin{gather*}
(x\otimes1)(x\otimes1)=x^2\otimes1,\\
(x^2\otimes S^{(2)}_1)(x\otimes1)=x^3\otimes S^{(3)}_1,\\
(x\otimes1)(x^2\otimes S^{(2)}_1)=x^3\otimes S^{(3)}_2,\\
(x^2\otimes S^{(2)}_1)(x^2\otimes S^{(2)}_1)=x^4\otimes S^{(4)}_1\wedge S^{(4)}_3
\end{gather*}
and so on.
\end{example}

\begin{theorem}\label{incl-excl}
The dg-algebra $(A_G,d)$ is a free resolution of the corresponding algebra with monomial relations $\k\langle x_1,\ldots,x_n\rangle/(g_1,\ldots,g_m)$.
\end{theorem}

\begin{proof}
Let us call a collection of divisors $S_1,\ldots, S_q$ of~$x_I$ \emph{indecomposable}, if each product $x_{i_k}x_{i_{k+1}}$ is contained in at least one of them. Then it is easy to see that $A$ is freely generated by elements $x_k\otimes 1$ and $x_I\otimes S_1\wedge\ldots\wedge S_q$ where $S_1,\ldots,S_q$ is an indecomposable collection of divisors of~$x_I$. Similarly, $A_G$ is freely generated by its basis elements $x_k\otimes 1$ and all elements $x_I\otimes S_1\wedge\ldots\wedge S_q$ where $S_1,\ldots,S_q$ is an indecomposable collection of divisors, each of which is a relation of~$R$.

Let us prove that $A_G$ provides a resolution for $R$. Since the differential $d$ only omits wedge factors but does not change the monomial, the chain complex $A_G$ is isomorphic to the direct sum of chain complexes $A_G^I$ spanned by the elements for which the first tensor factor is the given monomial $x_I\in\k\langle x_1,\ldots,x_n\rangle$. If $x_I$ is not divisible by any relation, the complex $A_G^I$ is concentrated in degree~$0$ and is spanned by~$x_I\otimes 1$. Thus, to prove the theorem, we should show that $A_G^I$ is acyclic whenever $x_I$ is divisible by some relation $g_i$.

Assume that there are exactly $k$ divisors of~$x_I$ which are relations of~$R$. We immediately see that the complex~$A_G^I$ is isomorphic to the inclusion--exclusion complex for the set $[k]$ 
\begin{equation}\label{InclExcl}
0\leftarrow \varnothing \leftarrow \bigoplus_{i=1}^{k} \{i\} \leftarrow
\bigoplus_{1\leq i<j\leq k} \{i,j\}\leftarrow \ldots
\leftarrow \{1,\ldots,k\}\leftarrow 0 
\end{equation}
(with the usual differential omitting elements). The latter one is acyclic whenever $k>0$, which completes the proof. 
\end{proof}

\begin{remark}
A similar construction works for commutative algebras as well, producing the corresponding homology groups. We shall not discuss it in detail here; the main idea is that one can make the symmetric groups act on the free algebras $A$ and $A_G$ from this section in such a way that the subalgebra of invariants is a free (super)commutative dg-algebra whose cohomology is the given monomial commutative algebra (acyclicity of the corresponding resolution can be derived from the acyclicity in the associative case, as subcomplexes of invariants of symmetric groups acting on the acyclic complexes have to be acyclic by the Maschke's theorem). All further constructions of the paper apply as well; we do not discuss any details or consequences here. It would be interesting to compare thus constructed resolutions with other known resolutions for monomial commutative algebras~\cite{Berglund,AlgMorse}.
\end{remark}

\subsection{Right module resolution for monomial algebras}\label{right_mod_alg}

Results of the previous section compute the bar homology of~$R$, since it coincides with the homology of the differential induced on the space of generators of a free resolution. This homology also coincides with the homology $\Tor^R_\ldot(\k,\k)$ of $R$, which is the homology of the complex of generators of a free $R$-module resolution of~$\k$. Let us explain how to materialize this statement via a free module resolution. We denote by $V_0$ the linear span of all $x_k\otimes 1$, and by $V_q, q>0$, the linear span of all elements $x_I\otimes S_1\wedge\ldots\wedge S_q$ as above. We shall construct a free resolution of the form
\begin{equation}\label{MonomRes}
\ldots\to V_q\otimes R\to V_{q-1}\otimes R\to\ldots\to V_1\otimes R\to V_0\otimes R\to R\to\k\to0. 
\end{equation}
It is enough to define boundary maps on the free module generators $V_q$, since boundary maps are morphisms of $R$-modules. First of all, we let $d_0\colon V_0\otimes R\to R$ be defined as $d_0(x_k\otimes 1)=x_k$. Assume that $q>0$ and let  $x_I\otimes S_1\wedge\ldots\wedge S_q\in V_q$. In the free algebra $A_G$, the differential~$d$ maps this element to a sum of elements corresponding to all possible omissions of $S_j$. If after the omission of~$S_j$ we still have an indecomposable covering, this summand survives in the differential. Otherwise, if after the omission of~$S_j$ the resulting covering is decomposable, and there exists a decomposition $x_I=x_Jx_K$ so that $S_1, \ldots, \hat{S_j}, \ldots,S_q$ form an indecomposable covering of $x_J$, then the corresponding summand of the differential becomes $(-1)^{j-1}(x_J\otimes S_1\wedge\ldots\wedge\hat{S_j}\wedge\ldots\wedge S_q)\otimes x_K\in V_{q-1}\otimes R$.

The following proposition is quite easy to prove, and we omit the details.
\begin{proposition}
The construction above provides a resolution of the trivial module by free~$R$-modules.
\end{proposition}

Readers familiar with the machinery of twisting cochains~\cite{twisting} may see our construction of the free right module resolution from the free dg-algebra resolution as a variation of the twisting cochain construction. More precisely, the differential of a generator in our free algebra resolution is a sum of products of generators; this provides the space of generators with a structure of an $\infty$-coalgebra, and the twisting cochain method applies. See~\cite{Lefevre-Hasegawa,Proute} for details in the case  of algebras, and \cite{DCV} for details in the operad case. 

\subsection{Homology classes for monomial algebras}\label{ass_minimal}

The inclusion-exclusion resolutions constructed above are in general not minimal. In this section, we give a description of generators of a minimal resolution for any monomial algebra. Our construction is a natural refinement of the inclusion-exclusion construction; we also identify it with the construction of Anick~\cite{Anick}, thus explaining one possible way to invent his results.

Let $R=\k\langle x_1,\ldots,x_n\rangle/(g_1,\ldots,g_m)$ be an algebra with $n$ generators and $m$ monomial relations. For the generators of a minimal resolution for the trivial module, one can take homology classes for the differential induced on the space of generators $(A_G)_{ab}=(A_G)_+/(A_G)_+^2$ of the algebra~$A_G$. We shall now give one possible set of representatives for these homology classes.  
 
Let us restrict ourselves to one of the chain complexes $(A_G^I)_{ab}$ corresponding to a particular monomial~$x_I$. We fix some ordering of the set of all relation divisors of~$f$; let those relations be $S_1<S_2<\ldots<S_m$. Then on the acyclic complex $A_G^I$ we have $m$ anticommuting differentials $\partial_1,\ldots,\partial_m$ where $\partial_p=\frac{\partial \phantom{S_p}}{\partial S_p}$, and the corresponding contracting homotopies $\imath_p(\cdot)=S_p\wedge\cdot$. Let us denote $\partial_{(p)}=\partial_1+\ldots+\partial_p$, so that the differential $d$ is nothing but~$\partial_{(m)}$. 

Our key observation is that for each~$p\le m$, 
 $$
H((A_G)_{ab}^I,\partial_{(p)})=H(\ldots(H((A_G)_{ab}^I,\partial_1),\ldots),\partial_p),
 $$
in other words, the homology of the total complex coincides with the iterated homology. This can be easily proved by induction on~$p$, since on each step the spectral sequence of the corresponding bicomplex degenerates at its page~$E_1$ (since computing $\partial_{(p)}^{-1}$ commutes with $\partial_{p+1}$). This observation results in the following statement, that can also be proved by induction on~$p$.

\begin{proposition}
For each $p\le m$, the homology $H((A_G)_{ab}^I,\partial_{(p)})$ has for representatives of all classes all monomials $M=x_I\otimes S_{i_1}\wedge\ldots\wedge S_{i_q}$ that satisfy the following properties:
\begin{itemize}
 \item[(i)] for all $j\le p$, the monomial $\partial_j(M)$ is either decomposable or equal to~$0$.
 \item[(ii)] for each $q'\le p$ such that $\partial_{q'}(M)=0$, there exists $q<q'$ for which $\imath_{q'}\partial_q (M)\ne0$ in $(A_G^I)_{ab}$ (i.e., this monomial is indecomposable).  
\end{itemize}
\end{proposition}

Setting $p=m$ in this result, we get the following description of the homology of~$(A_G)_{ab}^I$, that is the generators of the minimal resolution:
\begin{theorem}\label{min-resol}
The homology $H((A_G)_{ab}^I,d)$ has for representatives of all classes all monomials $M=f\otimes S_{i_1}\wedge\ldots\wedge S_{i_q}$ that satisfy the following properties:
\begin{itemize}
 \item[(i)] for all $j$, the monomial $\partial_j(M)$ is either decomposable or equal to~$0$.
 \item[(ii)] for each $q'$ such that $\partial_{q'}(M)=0$, there exists $q<q'$ for which $\imath_{q'}\partial_q (M)\ne0$ in $(A_G)_{ab}^I$ (i.e., this monomial is indecomposable). 
\end{itemize} 
\end{theorem}

There is a natural way to order all divisors of a given monomial~$x_I$, listing them according to their starting point. It turns out that for this ordering there is another elegant description for the representatives of homology classes. Let us recall the definition of Anick chains~\cite{Anick,Ufn}. Every chain is a monomial of the free algebra $\k\langle x_1,\ldots,x_n\rangle$. For $q\ge0$, $q$-chains and their tails are defined inductively as follows:
\begin{itemize}
\item[-] each generator $x_i$ is a $0$-chain; it coincides with its tail;
\item[-] each $q$-chain is a monomial $m$ equal to a product $nst$ where $t$ is the tail of $m$, and $ns$ is a $(q-1)$-chain whose tail is $s$;
\item[-] in the above decomposition, the product $st$ has exactly one divisor which is a relation of~$R$; this divisor is a right divisor of~$st$.
\end{itemize}
In other words, a $q$-chain is a monomial formed by linking one after another $q$ relations so that only neighbouring relations are linked, the first $(q-1)$ of them form a $(q-1)$-chain, and no proper left divisor is a $q$-chain. In our notation above, such a monomial $m$ corresponds to the generator $m\otimes S_1\wedge\ldots\wedge S_q$ where $S_1$, \ldots, $S_q$ are the relations we linked.

\begin{proposition}
For the above ordering of divisors, the representatives for homology classes from Theorem~\ref{min-resol} are exactly Anick chains.
\end{proposition}

\begin{proof}
Indeed, condition~(i) means that only neighbours are linked, and condition~(ii) means that no proper beginning of a $q$-chain forms a $q$-chain. 
\end{proof}

\subsection{Resolutions for general relations}\label{general}

In this section, we shall explain the machinery that transforms our resolution for a monomial replacement of the given algebra into a resolution for the original algebra. 

Let $\widetilde{R}=\k\langle x_1,\ldots,x_n\rangle/(\widetilde{G})$ be an algebra, and let $R=\k\langle x_1,\ldots,x_n\rangle/(G)$ be its monomial version, that is, $\widetilde{G}=\{\tilde{g}_1,\ldots,\tilde{g}_m\}$ is a Gr\"obner basis of relations, and $G=\{g_1,\ldots,g_m\}$ are the corresponding leading monomials. We have a free resolution $(A_G,d)$ for $R$, so that $H_\ldot(A_G,d)\simeq R$. Let $\pi$ (resp., $\widetilde{\pi}$) be the algebra homomorphism from $A_G$ to $R$ (resp., $\widetilde{R}$) that kills all generators of positive homological degree, and on elements of homological degree~$0$ is the canonical projection from $\k\langle x_1,\ldots,x_n\rangle$ to its quotient. Denote  by~$h$ the contracting homotopy for this resolution, so that $\left.(dh)\right|_{\ker d}= \Id-\pi$.

\begin{theorem}
There exists a ``deformed'' differential $D$ on~$A_G$ and a homotopy $H\colon\ker D\rightarrow A_G$ such that 
$H_\ldot(A_G,D)\simeq\widetilde{R}$, and $\left.(DH)\right|_{\ker D} = \Id-\widetilde{\pi}$.
\end{theorem}

\begin{proof}
We shall construct $D$ and $H$ simultaneously by induction. Let us introduce the following partial ordering of monomials in $A_G$: $f\otimes S_1\wedge \ldots\wedge S_q$ is, by definition, less than $f'\otimes S'_1\wedge \ldots\wedge S'_{q'}$ if the monomial~$f$ is less than $f'$ in the free algebra. This partial ordering suggests the following definition: for an element $u\in A_G$, its leading term $\hat{u}$ is the part of the expansion of~$u$ as a combination of basis elements where we keep only basis elements $f\otimes S_1\wedge \ldots\wedge S_q$ with maximal possible~$f$.

If $L$ is a homogeneous linear operator on $A_G$ of some fixed (homological) degree of homogeneity (like $D$, $H$, $d$, $h$), we denote by $L_k$ the operator $L$ acting on elements of homological degree~$k$. We shall define the operators $D$ and $H$ by induction: we define the pair $(D_{k+1},H_k)$ assuming that all previous pairs are defined. At each step, we shall also be proving that 
 $$
D(x)=d(\hat{x})+\text{lower terms}, \quad H(x)=h(\hat{x})+\text{lower terms},
 $$
where the words ``lower terms'' mean a linear combination of basis elements whose underlying monomial is smaller than the underlying monomial of~$\hat{x}$.
 
Basis of induction: $k=0$, so we have to define $D_1$ and $H_0$ (note that $D_0=0$ because there are no elements of negative homological degrees). In general, to define $D_l$, we should only consider the case when our element is a generator of~$A_G$, since in a dg-algebra the differential is defined by images of generators. For $l=1$, this means that we should consider the case where our generator corresponds to a leading monomial~$f=\lt(g)$ of some relation~$g$, and is of the form $f\otimes S$ where $S$ corresponds to the only divisor of $m$ which is a leading term, that is $f$ itself. We put 
 $$
D_1(f\otimes S)=\frac{1}{c_g}g,
 $$
where $c_g$ is the leading coefficient of~$g$. We see that  $$D_1(f\otimes S)=f+\text{lower terms},$$ as required. To define $H_0$, we use a yet another inductive argument, decreasing the monomials on which we want to define $H_0$. First of all, if a monomial~$f$ is not divisible by any of the leading terms of relations, we put~$H_0(f)=0$. Assume that $f$ is divisible by some leading terms of relations, and $S_1$, \ldots, $S_p$ are the corresponding divisors. Then on $A_G^f$ we can use $S_1\wedge\cdot$ as a homotopy, so $h_0(f)=f\otimes S_1$. We put 
 $$
H_0(f)= h_0(f) + H_0(f- D_1h_0(f)).
 $$
Here the leading term of $f-D_1h_0(f)$ is smaller than~$f$ (since we already know that the leading term of $D_1h_0(f)$ is $d_1h_0(f)=f$), so induction on the leading term applies. Note that by induction the leading term of $H_0(f)$ is $h_0(f)$. 

Suppose that $k>0$, that we know the pairs $(D_{l+1},H_l)$ for all $l<k$, and that in these degrees 
 $$
D(x)=d(\hat{x})+\text{lower terms}, \quad H(x)=h(\hat{x})+\text{lower terms}.
 $$
To define $D_{k+1}$, we should, as above, only consider the case of generators. In this case, we put
 $$
D_{k+1}(x)=d_{k+1}(x) - H_{k-1}D_kd_{k+1}(x).
 $$
The property $D_{k+1}(x)=d_{k+1}(\hat{x})+\text{lower terms}$ now easily follows by induction. To define $H_k$, we proceed in a way very similar to what we did for the induction basis. Assume that $u\in\ker D_k$, and that we know $H_k$ on all elements of $\ker D_k$ whose leading term is less than $\hat{u}$. Since $D_k(u)=d_k(\hat{u})+\text{lower terms}$, we see that $u\in\ker D_k$ implies $\hat{u}\in\ker d_k$. Then $h_k(\hat{u})$ is defined, and we put
 $$
H_k(u)= h_k(\hat{u}) + H_k(u- D_{k+1}h_k(\hat{u})).
 $$
Here $u-D_{k+1}h_k(\hat{u})\in\ker D_k$ and its leading term is smaller than~$\hat{u}$, so induction on the leading term applies (and it is easy to check that by induction $H_{k+1}(x)=h_{k+1}(\hat{x})+\text{lower terms}$). 

Let us check that the mappings $D$ and $H$ defined by these formulas satisfy, for each $k>0$, $D_kD_{k+1}=0$ and $\left.(D_{k+1}H_{k})\right|_{\ker D_{k}} = \Id-\widetilde{\pi}$. A computation checking that is somewhat similar to the way $D$ and $H$ were constructed. Let us prove both statements simultaneously by induction. If $k=0$, the first statement is obvious. Let us prove the second one and establish that $D_1H_0(f)=(\Id-\widetilde{\pi})(f)$ for each monomial $f$. Slightly rephrasing that, we shall prove that for each monomial $f$ we have $D_1H_0(f)=f-\overline{f}$ where $\overline{f}$ is the residue of $f$ modulo~$G$ \cite{DK}. We shall prove this statement by induction on $f$. If the monomial~$f$ is not divisible by any leading terms of relations, we have $H_0(f)=0=f-\overline{f}$. Let $f$ be divisible by leading terms $f_1$, \ldots, $f_p$, and let $S_1$, \ldots, $S_p$ be the corresponding divisors. We have $H_0(f)= h_0(f) + H_0(f- D_1h_0(f))$, so 
 $$
D_1H_0(f)= D_1h_0(f) + D_1H_0(f- D_1h_0(f)).
 $$
By induction, we may assume that
 $$
D_1H_0(f- D_1h_0(f))=f- D_1h_0(f)-\overline{(f- D_1h_0(f))}.
 $$
Also, 
 $$
D_1h_0(f)=D_1(f\otimes S_1)=\frac{1}{c_g}f',
 $$ 
 where $g$ is the relation with the leading monomial $f_1$, and $$\frac{1}{c_g}f'=\frac{1}{c_g}m_{f,f_1}(f)=f-r_g(f)$$ is the (normalized) result of substitution of~$g$ into~$f$ in the place described by~$S_1$.
Consequently, 
\begin{multline*}
D_1H_0(f)= f-r_g(f)+ \left((f- D_1h_0(f))-\overline{(f- D_1h_0(f))}\right)=\\=f-r_g(f)+(r_g(f)-\overline{r_g(f)})=f-\overline{r_g(f)}=f-\overline{f}, 
\end{multline*}
since the residue does not depend on a particular choice of reductions.

Assume that $k>0$, and that our statement is true for all $l<k$. We have 
 $$D_{k}D_{k+1}(x)=0$$ 
since
\begin{multline*}
D_{k}D_{k+1}(x)= D_k(d_{k+1}(x) - H_{k-1}D_kd_{k+1}(x))= \\=
D_kd_{k+1}(x) - D_kH_{k-1}D_kd_{k+1}(x)= D_kd_{k+1}(x) - D_kd_{k+1}(x)= 0,
\end{multline*}
because $D_kd_{k+1}k\in\ker D_{k-1}$, and so $D_kH_{k-1}(D_k(y))=D_k(y)$ by induction.
Also, for $u\in\ker D_k$ we have
 $$
D_{k+1}H_k(u)= D_{k+1}h_k(\hat{u}) + D_{k+1}H_k(u- D_{k+1}h_k(\hat{u})),
 $$
and by the induction on $\hat{u}$ we may assume that 
 $$
D_{k+1}H_k(u- D_{k+1}h_k(\hat{u}))=u- D_{k+1}h_k(\hat{u})
 $$
(on elements of positive homological degree, $\widetilde{\pi}=0$), so
 $$
D_{k+1}H_k(u)= D_{k+1}h_k(\hat{u}) + u- D_{k+1}h_k(\hat{u})=u, 
 $$
which is exactly what we need. 
\end{proof}

\begin{remark}
\begin{itemize}
 \item The above construction works without a problem for every finitely generated algebra with a Gr\"obner basis of relations, provided that the monomials in the free algebra form a well-ordered set; in that case one can be sure that inductive definitions provide well-defined objects. In the case when the relations are homogeneous, the resolution that we obtain comes with an additional internal grading corresponding to the grading on the quotient algebra that we are trying to resolve. 
 \item A similar construction for the case of free resolution of the trivial module over the given augmented algebra was described by Anick~\cite{Anick} and Kobayashi~\cite{Kobayashi} (see also Lambe~\cite{Lambe}). There are several other ways to construct free resolutions, see e.g. \cite{Brown, Cohen, AlgMorse, Skoldberg} where the idea is to start from the bar complex, select there candidates that we want to be the generators of a smaller free resolution, and construct a contraction of the bar complex on that subcomplex. We shall discuss analogues of these constructions beyond the case of algebras (e.g. for operads) elsewhere.
\end{itemize}
\end{remark}

\subsection{An instructional example: an algebra which resembles the $BV$ operad} \label{like-BV}

In this section, we shall consider a particular example of an associative algebra which we find quite useful.

\begin{definition}
The algebra $\sfA$ is an associative algebra with two generators $x,y$ and two relations
\begin{gather*}
 y^2=0, \\
x^2y+xyx+yx^2=0.
\end{gather*}
\end{definition}

From the Gr\"obner bases viewpoint, this algebra shares many similar features with the operad of Batalin--Vilkovisky algebras (which we shall consider in Section~\ref{BV}): it has relations of degrees $2$ and~$3$, and a Gr\"obner basis is obtained from them by adjoining a relation of degree~$4$, the resolution for a monomial replacement of this algebra is automatically minimal, but deformation of the differential to incorporate lower terms leads to many cancellations etc. However, the relevant computations here are less demanding, so we hope it would be easier for our readers to get a flavour of our approach from this example.

The following result can be checked by a direct computation.
\begin{proposition}
Let $x>y$. For the lexicographic ordering of monomials, the Gr\"obner basis for the algebra $\sfA$ is given by
\begin{gather*}
 y^2=0, \\
x^2y+xyx+yx^2=0,\\
xyxy-yxyx=0.
\end{gather*}
\end{proposition}

This immediately tranlates into results for the monomial version of our algebra, that is the algebra with relations
\begin{gather*}
 y^2=0, \\
x^2y=0,\\
xyxy=0.
\end{gather*}

\begin{proposition}\label{BV-like-alg-monomial}
The homology of the monomial version of~$\sfA$ is represented by the classes  
\begin{gather*}
 x,y\in\Tor_1,\\
 y^{k+2}\in\Tor_{k+2} (k\ge0), \\
x^2y^{k+1}\in\Tor_{k+2} (k\ge0),\\
(xy)^{l+1}xy^{k+1}\in\Tor_{k+l+2} (k,l\ge0),\\
x(xy)^{l+1}xy^{k+1}\in\Tor_{k+l+3} (k,l\ge0).
\end{gather*}
\end{proposition}

\begin{proof}
These elements are precisely the generators of the free resolution, which in this case is minimal because only neighbour relations overlap in them. 
\end{proof}

\begin{theorem}\label{BV-like-alg}
The homology $\Tor^\sfA_\ldot(\k,\k)$ of the algebra $\sfA$ is represented by the classes  
\begin{gather*}
 x,y\in\Tor_1,\\
 y^{k+2}\in\Tor_{k+2} (k\ge0), \\
x(xy)^{k+1}\in\Tor_{k+2} (k\ge0).
\end{gather*}
\end{theorem}

\begin{proof}
One can check that for the deformed differential~$D$ we have
\begin{gather*}
D(x^2y^{k+2})=xyxy^{k+1}+\text{lower terms},\\
D(x(xy)^{l+1}xy^{k+2})=(xy)^{l+2}xy^{k+1}+\text{lower terms},
\end{gather*}
which leaves from the cocycles described in Proposition~\ref{BV-like-alg-monomial} only
$y^{k+2}$ for all $k\ge0$, $x^2y^{k+1}$ for $k=0$ (which coincides with $x(xy)^{k+1}$ for $k=0$),
and $x(xy)^{l+1}xy^{k+1}$ for $k=0$ (which coincides with $x(xy)^{k+1}$ for $k=l+1$). The images of
these elements under the differential are killed by the augmentation, and the theorem follows.
\end{proof}

In fact, it is possible to exhibit a minimal resolution of $\sfA$. Let us denote by $a_k$ the homology class represented
by $y^k$, and by $b_k$ the homology class represented by $x(xy)^{k-1}$; these classes form a basis of the two-dimensional
space $V_k=\Tor^\sfA_k(\k,\k)$. We leave it to the reader to verify that the following formulas define a resolution of $\k$ by free $\sfA$-modules $V_k\otimes \sfA$:
\begin{gather*}
D(a_k) = a_{k-1}\otimes y \\
D(b_k) = 
\left\{
\begin{aligned}
&b_{k-1}\otimes xy  - a_{k-1}\otimes x^k, \qquad\qquad k\equiv 0\pmod{3},\\
&b_{k-1}\otimes yx  + a_{k-1}\otimes x^k, \qquad\qquad k\equiv 1\pmod{3},\\
&b_{k-1}\otimes (xy + yx)  + a_{k-1}\otimes x^k, \quad k\equiv 2\pmod{3}.\\
\end{aligned}
\right.
\end{gather*} 

\begin{remark}
Another interesting feature of the algebra~$\sfA$ we considered here is that it provides a counterexample to a result of Farkas~\cite{Farkas} who gave a very simple formula for differentials in the Anick resolution in terms of the Gr\"obner basis of~$\sfA$. One can easily check that if we apply the Farkas' differential to the chain $x^2y^2$ once, we get the element $x^2y\otimes y$, and if we apply the differential once again, we get $x\otimes yxy-y\otimes xyx-y\otimes yx^2$, a nonzero element of the resolution. Therefore, Farkas' formula for the differential is generally incorrect, and there seem to be no formula that is as simple as the one he suggested. For the contrast, there are various algorithmic approaches to computing maps in the Anick resolution, see e.g.~\cite{AlgorResol1,AlgorResol}.
\end{remark}

\section{Operads}\label{operads}

\subsection{Shuffle operads}

For information on symmetric and nonsymmetric operads, we refer the reader to the monograph~\cite{MSS}, for information on shuffle operads and Gr\"obner bases for operads --- to our paper~\cite{DK}. Throughout this paper by an operad we mean a shuffle operad, unless otherwise specified; there is no machinery of Gr\"obner bases available in the symmetric case, so we have to sacrifice the symmetric groups action. However, once we move to the shuffle category, all constructions of the previous section work perfectly fine (in fact, the construction for algebras is a particular case of the construction below, applied to algebras thought of as operads concentrated in arity~$1$).

For the monoidal category of shuffle operads, it is possible to define the bar complex of an augmented operad~$\calO$. The bar complex $\mathbf{B}^\udot(\calO)$ is a dg-cooperad freely generated by the degree shift $\calO_+[1]$ of the augmentation ideal of $\calO$; the differential comes from operadic compositions in~$\calO$. Similarly, for a cooperad $\calQ$, it is possible to define the cobar complex~$\Omega^\udot(\calQ)$, which is a dg-operad freely generated by~$\calQ_+[-1]$, with the appropriate differential. The bar-cobar construction $\Omega^\udot(\mathbf{B}^\udot(\calO))$ gives a free resolution of~$\calO$. This can be proved in a rather standard way, similarly to known proofs in the case of operads, properads etc.~\cite{GK,F,ValPROP}. The general homotopical algebra philosophy mentioned in the introduction is applicable in the case of operads as well; various checks and justifications needed to ensure that are quite standard and similar to the ones available in the literature; we refer the reader to \cite{BM, Fresse, Harper1, MarklModel, Rezk, Spitzweck} where symmetric operads are handled. 

It is important to recall here that the forgetful functor ${}^f\colon\calP\to\calP^f$ from the category of symmetric operads to the category of shuffle operads is monoidal~\cite{DK}, which easily implies that for a symmetric operad~$\calP$, we have 
 $$
\mathbf{B}^\udot(\calP)^f\simeq\mathbf{B}^\udot(\calP^f), 
 $$
that is the (symmetric) bar complex of $\calP$ is naturally identified, as a shuffle dg-cooperad, with the (shuffle) bar complex of~$\calP^f$. Thus, our approach would enable us to compute the homology even in the symmetric case, only without information on the symmetric groups action. 

\subsection{Tree monomials}

Let us recall tree combinatorics used to describe monomials in shuffle operads. See~\cite{DK} for more details.

Basis elements of the free operad are represented by (decorated) trees. A (rooted) \emph{tree} is a non-empty connected directed graph $T$ of genus~$0$ for which each vertex has at least one incoming edge and exactly one outgoing edge. Some edges of a tree might be bounded by a vertex at one end only. Such edges are called \emph{external}. Each tree should have exactly one outgoing external edge, its \emph{output}. The endpoint of this edge which is a vertex of our tree is called the \emph{root} of the tree. The endpoints of incoming external edges which are not vertices of our tree are called \emph{leaves}.

Each tree with~$n$ leaves should be (bijectively) labelled by~$[n]$. For each vertex $v$ of a tree, the edges going in and out of $v$ will be referred to as inputs and outputs at~$v$. A tree with a single vertex is called a \emph{corolla}. There is also a tree with a single input and no vertices called the \emph{degenerate} tree. Trees are originally considered as abstract graphs but to work with them we would need some particular representatives that we now going to describe.

For a tree with labelled leaves, its canonical planar representative is defined as follows. In general, an embedding of a (rooted) tree in the plane is determined by an ordering of inputs for each vertex. To compare two inputs of a vertex~$v$, we find the minimal leaves that one can reach from~$v$ via the corresponding input. The input for which the minimal leaf is smaller is considered to be less than the other one. Note that this choice of a representative is essentially the same one as we already made when we identified symmetric compositions with shuffle compositions.

Let us introduce an explicit realisation of the free operad generated by a collection $\calM$. The basis of this operad will be indexed by planar representative of trees with decorations of all vertices. First of all, the simplest possible tree is the degenerate tree; it corresponds to the unit of our operad. The second simplest type of trees is given by corollas. We shall fix a basis~$B^\calM$ of $\calM$ and decorate the vertex of each corolla with a basis element; for a corolla with $n$ inputs, the corresponding element should belong to the basis of~$\calV(n)$. The basis for whole free operad consists of all planar representatives of trees built from these corollas 
(explicitly, one starts with this collection of corollas, defines compositions of trees in terms of grafting, and then considers all trees obtained from corollas by iterated shuffle compositions). We shall refer to elements of this basis as \emph{tree monomials}.

There are two standard ways to think of elements of an operad defined by generators and relations: using either tree monomials or operations. Our approach is somewhere in the middle: we prefer (and strongly encourage the reader) to think of tree monomials, but to write formulas required for definitions and proofs we prefer the language of operations since it makes things more compact. 

Let us give an example of how to translate between these two languages. Let $\calO=\calF_\calM$ be the free operad for which the only nonzero component of $\calM$ is $\calM(2)$, and the basis of $\calM(2)$ is given by 
 $$
2\includegraphics[scale=0.8]{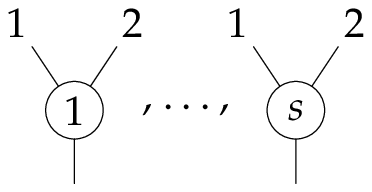}
 $$
Then the basis of $\calF_\calM(3)$ is given by the tree monomials
 $$
\includegraphics[scale=0.8]{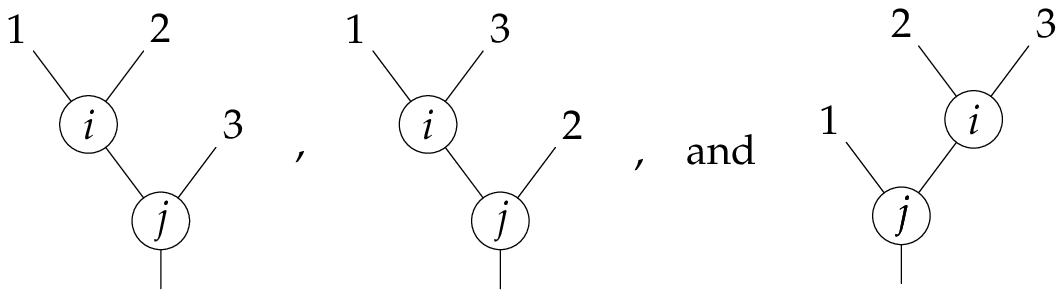}
 $$
with $1\le i,j\le s$.

If we assume that the $j^\text{th}$ corolla corresponds to the operation $$\mu_j\colon a_1,a_2\mapsto\mu_j(a_1,a_2),$$ then the above tree monomials correspond to operations
 $$
\mu_j(\mu_i(a_1,a_2),a_3), \quad \mu_j(\mu_i(a_1,a_3),a_2),\quad \text{and} \quad \mu_j(a_1,\mu_i(a_2,a_3))
 $$
respectively. 

Take a tree monomial~$\alpha\in\calF_\calM$. If we forget the labels of its vertices and its leaves, we get a planar tree. We shall refer to this planar tree as \emph{the underlying tree of~$\alpha$}. Divisors of $\alpha$ in the free operad correspond to a special kind of subgraphs of its underlying tree. Allowed subgraphs contain, together with each vertex, all its incoming and outgoing edges (but not necessarily other endpoints of these edges). Throughout this paper we consider only this kind of subgraphs, and we refer to them as subtrees hoping that it does not lead to any confusion. 

Clearly, a subtree~$T'$ of every tree~$T$ is a tree itself. Let us define the tree monomial $\alpha'$ corresponding to~$T'$. To label vertices of~$T'$, we recall the labels of its vertices in~$\alpha$.  We immediately observe that these labels match the restriction labels of a tree monomial should have: each vertex has the same number of inputs as it had in the original tree, so for a vertex with $n$ inputs its label does belong to the basis of~$\calM(n)$. To label leaves of~$T'$, note that each such leaf is either a leaf of~$T$, or is an output of some vertex of~$T$. This allows us to assign to each leaf~$l'$ of~$T'$ a leaf~$l$ of~$T$: if $l'$ is a leaf of~$T$, put $l=l'$, otherwise let~$l$ be the smallest leaf of $T$ that can be reached through~$l'$. We then number the leaves according to these ``smallest descendants'': the leaf with the smallest 
possible descendant gets the label~$1$, the second smallest~--- the label~$2$ etc. 

\begin{example}
Let us consider the tree monomial 
 $$ 
\includegraphics[scale=0.9]{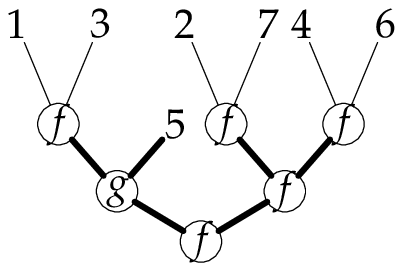}
 $$
in the free operad with two binary generators labelled $a$ and $b$; in the language of operations, it corresponds to the expression $$f(g(f(a_1,a_3),a_5),f(f(a_2,a_7),f(a_4,a_6))).$$
One of its subtrees (indicated by bold lines in the figure) produces the tree monomial
 $$ 
\includegraphics[scale=0.9]{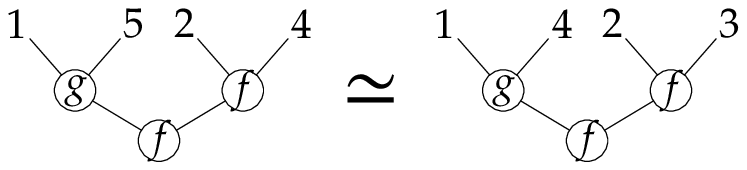}
 $$
(in the language of operations $f(g(a_1,a_5),f(a_2,a_4))\simeq f(g(a_1,a_4),f(a_2,a_3))$), where the labels on the left come from minimal leaves, as explained above. Note that even though in this example the subtree shares the root vertex with the original tree, in general it is not required.
\end{example}

For two tree monomials $\alpha$, $\beta$ in the free operad $\calF_\calM$, we say that \emph{$\alpha$ is divisible by $\beta$}, if there exists a subtree of the underlying tree of $\alpha$ for which the corresponding tree monomial $\alpha'$ is equal to~$\beta$.

Let us introduce an example of an operad which will be used to illustrate methods of this section. It was defined and studied by Markl and Remm in~\cite{MR}.
\begin{definition}
The anti-associative operad $\widetilde{\As}$ 
is the nonsymmetric operad with one generator $f(\textrm{-},\textrm{-})\in\widetilde{\As}(2)$ and one relation 
\begin{equation}\label{AntiAsRel}
f(f(\textrm{-},\textrm{-}),\textrm{-})+f(\textrm{-},f(\textrm{-},\textrm{-}))=0. 
\end{equation}
 \end{definition}

For the path-lexicographic ordering, the element $f(f(f(\textrm{-},\textrm{-}),\textrm{-}),\textrm{-})$ is a small common multiple of the leading monomial with itself, and the corresponding S-polynomial is equal to $2f(\textrm{-},f(\textrm{-},f(\textrm{-},\textrm{-})))$. These relations together already imply that $\calA(k)=0$ for $k\ge4$, so we have the following
\begin{proposition}
Two relations $f(f(\textrm{-},\textrm{-}),\textrm{-})+f(\textrm{-},f(\textrm{-},\textrm{-}))$ and $f(\textrm{-},f(\textrm{-},f(\textrm{-},\textrm{-})))$ form a Gr\"obner basis of relations that define~$\widetilde{\As}$.
\end{proposition}
The corresponding operad with monomial relations is defined by relations $f(f(\textrm{-},\textrm{-}),\textrm{-})=0$ and $f(\textrm{-},f(\textrm{-},f(\textrm{-},\textrm{-})))=0$,
and has a monomial basis $\{\id, f, g:=f(\textrm{-},f(\textrm{-},\textrm{-}))\}$.

\subsection*{Operads in the differential graded setting} 

The above description of the free shuffle operad works almost literally when we work with operads whose components are chain complexes (as opposed to vector spaces), and the symmetric monoidal structure on the corresponding category involves signs. The only difference is that every tree monomial should carry an ordering of its internal vertices, so that two different orderings contribute appropriate signs. In this section, we give an examples of a shuffle dg-operad that should help a reader to understand the graded case better; it is very close to the (ungraded) anti-associative operad which we discuss throughout the paper. Lie the anti-associative operad, it is also introduced in~\cite{MR}.

\begin{definition}
The odd $(2k+1)$-associative operad is a nonsymmetric operad with one generator~$\mu$ of arity~$n=2k+1$ and homological degree~$2l+1$, and relations $\mu\circ_p\mu=\mu\circ_{2k+1}\mu$ for all~$p\le 2k$. 
\end{definition}

Let us show that the Buchberger algorithm for operads from~\cite{DK} discovers a cubic relation in the Gr\"obner basis for this operad, thus showing that this operad fails to be PBW. We use the path-lexicographic ordering of monomials. 

From the common multiple $(\mu\circ_1\mu)\circ_1\mu$ of the leading term $\mu\circ_1\mu$ with itself, we compute the S-polynomial
 $$
(\mu\circ_n\mu)\circ_1\mu-\mu\circ_1(\mu\circ_n\mu). 
 $$
We can perform the following chain of reductions (with leading monomials underlined):
\begin{multline*}
(\mu\circ_n\mu)\circ_1\mu-\underline{\mu\circ_1(\mu\circ_n\mu)}=
(\mu\circ_n\mu)\circ_1\mu-\underline{(\mu\circ_1\mu)\circ_n\mu}\mapsto\\ \mapsto
\underline{(\mu\circ_n\mu)\circ_1\mu}-(\mu\circ_n\mu)\circ_n\mu=
-\underline{(\mu\circ_1\mu)\circ_{2n-1}\mu}-(\mu\circ_n\mu)\circ_n\mu\mapsto\\ \mapsto
-(\mu\circ_n\mu)\circ_{2n-1}\mu-\underline{(\mu\circ_n\mu)\circ_n\mu}=
-(\mu\circ_n\mu)\circ_{2n-1}\mu-\underline{\mu\circ_n(\mu\circ_1\mu)}\mapsto\\ \mapsto
-(\mu\circ_n\mu)\circ_{2n-1}\mu-\mu\circ_n(\mu\circ_n\mu)=
-2(\mu\circ_n\mu)\circ_{2n-1}\mu.
\end{multline*}
Note that we used the formula $(\mu\circ_n\mu)\circ_1\mu=-(\mu\circ_1\mu)\circ_{2n-1}\mu$ which reflect the fact that the operation $\mu$ is of odd homological degree.

The monomial $(\mu\circ_n\mu)\circ_{2n-1}\mu$ cannot be reduced further, and we recover the relation $(\mu\circ_n\mu)\circ_{2n-1}\mu=0$ discovered in~\cite{MR}. Furthermore, we arrive at the following proposition (note the similarity with the computation of the Gr\"obner basis for the operad~$\AntiCom$ in~\cite{DK}).

\begin{proposition}
Elements $\mu\circ_p\mu-\mu\circ_{2k+1}\mu$ with $1\le p\le 2k$ and $(\mu\circ_n\mu)\circ_{2n-1}\mu$ form a Gr\"obner basis for the operad of odd $(2k+1)$-associative algebras.
\end{proposition}

\subsection{Inclusion-exclusion resolution}\label{oper_monom}

Let us construct a free resolution for an arbitrary operad with monomial relations. Assume that the operad~$\calO$ is generated by a collection of finite sets $\calM=\{\calM(n)\}$, with $m$ monomial relations $g_1,\ldots,g_m$ (this means that every tree monomial divisible by any of the relations is equal to zero), $\calO=\calF_\calM/(g_1,\ldots,g_m)$. We denote by $\calA(T,q)$ the vector space with the basis consisting of all elements $T\otimes S_1\wedge S_2\wedge\ldots\wedge S_q$ where $T$ is a tree monomial from the free shuffle operad~$\calF_M$, and $S_1,\ldots,S_q$, $q\ge0$, are tree divisors of~$T$. 

The differential $d$ with
\begin{equation}
d(T\otimes S_1\wedge \ldots\wedge S_q)= \sum_{l} (-1)^{l-1} T\otimes S_1\wedge\ldots\wedge\hat{S_l}\wedge\ldots\wedge S_q  
\end{equation}
makes the graded vector space
\begin{equation}
\calA(n)=\bigoplus_{T\text{ with $n$ leaves}}\bigoplus_q \calA(T,q) 
\end{equation}
into a chain complex. There is also a natural operad structure on the collection $\calA=\{\calA(n)\}$; the operadic composition composes the trees, and computes the wedge product of divisors (using, as in the case of algebras the identification $\imath$ of tree divisors of $T$ and $T'$ with the corresponding divisors of their composition). Overall, we defined a shuffle dg-operad.

Let $\calG=\{g_1,\ldots,g_m\}$ be the set of relations of our operad. The dg-operad $(\calA_\calG,d)$ is spanned by the elements $T\otimes S_1\wedge\ldots\wedge S_k$ where for each~$j$ the divisor of $T$ corresponding to $S_j$ is a relation. The differential~$d$ is the restriction of the differential defined above. Informally, an element of the operad $\calA_\calG$ is a tree with some distinguished divisors that are relations from the given set.

The following theorem is proved analogously to its counterpart for associative algebras, Theorem~\ref{incl-excl} above.

\begin{theorem}
The dg-operad $(\calA_\calG,d)$ is a free resolution of the corresponding operad with monomial relations $\calO=\calF_\calM/(\calG)$.
\end{theorem}

\subsection{Right module resolution for monomial operads}\label{right_mod_oper}

Similarly to how it is done in Section~\ref{right_mod_alg}, it is possible to prove the following 

\begin{theorem}
Let $\calO$ be a monomial operad, and let us denote by $\calV_q$ the collection of the generators of the free resolution from the previous
section of homological degree~$q$. Then there exists an exact sequence of collections
\begin{equation}\label{MonomResOper}
\ldots\to \calV_q\circ\calO\to\calV_{q-1}\circ\calO\to\ldots\to\calV_1\circ\calO\to\calV_0\circ\calO\to\calO\to\k\to0. 
\end{equation}
\end{theorem}

Let us describe some low degree maps of the resolution for the monomial version of the anti-associative operad. We use the following notation for the low degree basis elements: $\alpha\in\calV_0(2)$ is the element corresponding to $f$, $\beta\in\calV_1(3)$ and $\gamma\in\calV_1(4)$ are the elements corresponding to $f(f(\textrm{-},\textrm{-}),\textrm{-})$ and $f(\textrm{-},f(\textrm{-},f(\textrm{-},\textrm{-})))$ respectively, and $\omega\in\calV_2(4)$ corresponds to the small common multiple we discussed earlier (the overlap of two copies of $f(f(\textrm{-},\textrm{-}),\textrm{-})$); there are other elements in $\calV_2$, but we shall use only this one in our example. 

The following proposition is straightforward; we encourage our readers to perform the computations themselves to get familiar with our approach.

\begin{proposition}\label{AntiAssocMonomialResol}
We have 
\begin{gather*}
d_0(\alpha(\textrm{-},\textrm{-}))=f(\textrm{-},\textrm{-}),\\
h_0(f(\textrm{-},\textrm{-}))=\alpha(\textrm{-},\textrm{-}),\\
d_1(\beta(\textrm{-},\textrm{-},\textrm{-}))=\alpha(f(\textrm{-},\textrm{-}),\textrm{-}),\\
d_1(\gamma(\textrm{-},\textrm{-},\textrm{-},\textrm{-}))=\alpha(\textrm{-},g(\textrm{-},\textrm{-},\textrm{-})),\\
h_1(\alpha(\textrm{-},g(\textrm{-},\textrm{-},\textrm{-})))=\gamma(\textrm{-},\textrm{-},\textrm{-},\textrm{-}),\\
h_1(\alpha(f(\textrm{-},\textrm{-}),f(\textrm{-},\textrm{-})))=\beta(\textrm{-},\textrm{-},f(\textrm{-},\textrm{-})),\\
h_1(\alpha(g(\textrm{-},\textrm{-},\textrm{-}),\textrm{-}))=\beta(\textrm{-},f(\textrm{-},\textrm{-}),\textrm{-}),\\
d_2(\omega(\textrm{-},\textrm{-},\textrm{-},\textrm{-}))=\beta(f(\textrm{-},\textrm{-}),\textrm{-},\textrm{-}). 
\end{gather*}
\end{proposition}

\subsection{Homology classes}\label{oper_minimal}

Also, one can obtain representatives for homology classes in exactly the same way as for associative algebras. Let us choose a tree monomial~$T$, and work with the inclusion-exclusion complex $\calA_\calG(T):=\bigoplus_q \calA_\calG(T,q)$. We fix some ordering of the set of all relation divisors of~$T$; let those relations be $S_1<S_2<\ldots<S_m$. On the acyclic complex $\calA_\calG(T)$ we have $m$ anticommuting differentials $\partial_1,\ldots,\partial_m$ where $\partial_p=\frac{\partial \phantom{S_p}}{\partial S_p}$, and the corresponding contracting homotopies $\imath_p(\cdot)=S_p\wedge\cdot$. Let us denote $\partial_{(p)}=\partial_1+\ldots+\partial_p$, so that the differential $d$ is nothing but~$\partial_{(m)}$. 

The following theorem is proved analogously to its counterpart for associative algebras, Theorem~\ref{min-resol} above.

\begin{theorem}\label{min-resol-oper}
The homology $H(\calA_\calG(T)_{ab},d)$ has for representatives of all classes monomials $M=T\otimes S_{i_1}\wedge\ldots\wedge S_{i_q}$ that satisfy the following properties:
\begin{itemize}
 \item[(i)] for all $j$, the monomial $\partial_j(M)$ is either decomposable or equal to~$0$.
 \item[(ii)] for each $q'$ such that $\partial_{q'}(M)=0$, there exists $q<q'$ for which $\imath_{q'}\partial_q (M)\ne0$ in $\calA_\calG(T)_{ab}$ (i.e., this monomial is indecomposable). 
\end{itemize} 
\end{theorem}

Unlike the situation for monomials in associative algebras, there is no natural ordering of divisors for a tree monomial in the free shuffle operad  (since ``trees grow in several different directions''), so there is no description of representatives in a manner as elegant as it is in the case of Anick resolution. However, in some cases it is possible to make use of it. We shall discuss some applications below, when dealing with particular examples.

Using results of this section, we can compute representatives for low degree homology classes of the monomial version of the anti-associative operad. Our results on dimensions of components for the corresponding minimal resolution $\calR$ are summarised in the following 
\begin{proposition}
We have
\begin{gather*}
\dim\calR(2)_0=1,\\
\dim\calR(3)_1=1,\\
\dim\calR(4)_1=\dim\calR(4)_2=1,\\
\dim\calR(5)_2=5, \dim\calR(5)_3=1,\\
\dim\calR(6)_3=15, \dim\calR(6)_4=1,\\
\dim\calR(7)_3=4, \dim\calR(7)_4=35, \dim\calR(7)_5=1.
\end{gather*} 
\end{proposition}

\subsection{Resolutions for general relations}\label{AntiAsExample}

Let $\widetilde{\calO}=\calF_\calM/(\widetilde{\calG})$ be an operad, and let $\calO=\calF_\calM/(\calG)$ be its monomial version, that is, $\widetilde{\calG}$ is a Gr\"obner basis of relations, and $\calG$ consists of all leading monomials of~$\widetilde{\calG}$. In Section~\ref{oper_monom}, we defined a free resolution $(\calA_\calG,d)$ for $\calO$, so that $H_\ldot(\calA_\calG,d)\simeq\calO$. Let $\pi$ (resp., $\widetilde{\pi}$) be the canonical homomorphism from $\calA_\calG$ to~$\calO$ that kills all generators of positive homological degree, and on elements of homological degree~$0$ is the canonical projection from $\calF_\calM$ to its quotient. Denote  by~$h$ the contracting homotopy for this resolution, so that $\left.(dh)\right|_{\ker d}= \Id-\pi$.

Similarly to how it is proved in Section~\ref{general}, we obtain the following

\begin{theorem}
There exists a ``deformed'' differential $D$ on~$\calA_\calG$ and a homotopy $H\colon\ker D\rightarrow\calA_\calG$ such that 
$H_\ldot(\calA_\calG,D)\simeq\widetilde{\calO}$, and $\left.(DH)\right|_{\ker D} = \Id-\widetilde{\pi}$.
\end{theorem}

Let us compute some low degree maps of the resolution for the anti-associative operad. We shall cheat a little bit and deform the right module resolution, not the dg-operad resolution, as the former is smaller and so the amount of computations is not excessive. Results of Proposition~\ref{AntiAssocMonomialResol} can be used to compute the deformed maps as follows.

\begin{proposition}
We have
\begin{gather*}
D_0(\alpha(\textrm{-},\textrm{-}))=f(\textrm{-},\textrm{-}), \\
H_0(f(\textrm{-},\textrm{-}))=\alpha(\textrm{-},\textrm{-}),\\
D_1(\beta(\textrm{-},\textrm{-},\textrm{-}))=\alpha(f(\textrm{-},\textrm{-}),\textrm{-})+\alpha(\textrm{-},f(\textrm{-},\textrm{-})),\\
D_1(\gamma(\textrm{-},\textrm{-},\textrm{-},\textrm{-}))=\alpha(\textrm{-},g(\textrm{-},\textrm{-},\textrm{-})),\\
H_1(\alpha(\textrm{-},g(\textrm{-},\textrm{-},\textrm{-})))=\gamma(\textrm{-},\textrm{-},\textrm{-},\textrm{-}),\\
H_1(\alpha(f(\textrm{-},\textrm{-}),f(\textrm{-},\textrm{-})))=\beta(\textrm{-},\textrm{-},f(\textrm{-},\textrm{-}))-\gamma(\textrm{-},\textrm{-},\textrm{-},\textrm{-}),\\
H_1(\alpha(g(\textrm{-},\textrm{-},\textrm{-}),\textrm{-}))=\beta(\textrm{-},f(\textrm{-},\textrm{-}),\textrm{-})+\gamma(\textrm{-},\textrm{-},\textrm{-},\textrm{-}),\\
D_2(\omega(\textrm{-},\textrm{-},\textrm{-},\textrm{-}))=\beta(f(\textrm{-},\textrm{-}),\textrm{-},\textrm{-})+\beta(\textrm{-},f(\textrm{-},\textrm{-}),\textrm{-})-\beta(\textrm{-},\textrm{-},f(\textrm{-},\textrm{-}))+2\gamma(\textrm{-},\textrm{-},\textrm{-},\textrm{-}).
\end{gather*}
\end{proposition}

\begin{proof}
Formulas for $D_0$ and $H_0$ are obvious. For $D_1$ and $H_1$, the computation goes as follows:
\begin{multline*}
D_1(\beta(\textrm{-},\textrm{-},\textrm{-}))=\alpha(f(\textrm{-},\textrm{-}),\textrm{-})-H_0D_0(\alpha(f(\textrm{-},\textrm{-}),\textrm{-}))=\\=
\alpha(f(\textrm{-},\textrm{-}),\textrm{-})-H_0(f(f(\textrm{-},\textrm{-}),\textrm{-}))=\alpha(f(\textrm{-},\textrm{-}),\textrm{-})+H_0(f(\textrm{-},f(\textrm{-},\textrm{-})))=\\=\alpha(f(\textrm{-},\textrm{-}),\textrm{-})+\alpha(\textrm{-},f(\textrm{-},\textrm{-})),
\end{multline*}
 $$
D_1(\gamma(\textrm{-},\textrm{-},\textrm{-},\textrm{-}))=\alpha(\textrm{-},g(\textrm{-},\textrm{-},\textrm{-}))-H_0D_0(\alpha(\textrm{-},g(\textrm{-},\textrm{-},\textrm{-})))=\alpha(\textrm{-},g(\textrm{-},\textrm{-},\textrm{-})),
 $$
 $$ 
H_1(\alpha(\textrm{-},g(\textrm{-},\textrm{-},\textrm{-})))=\gamma(\textrm{-},\textrm{-},\textrm{-},\textrm{-})+H_1(\alpha(\textrm{-},g(\textrm{-},\textrm{-},\textrm{-}))-D_1\gamma(\textrm{-},\textrm{-},\textrm{-},\textrm{-}))=\gamma(\textrm{-},\textrm{-},\textrm{-},\textrm{-}),
 $$
\begin{multline*}
H_1(\alpha(f(\textrm{-},\textrm{-}),f(\textrm{-},\textrm{-})))=\\=\beta(\textrm{-},\textrm{-},f(\textrm{-},\textrm{-}))+H_1(\alpha(f(\textrm{-},\textrm{-}),f(\textrm{-},\textrm{-}))-D_1(\beta(\textrm{-},\textrm{-},f(\textrm{-},\textrm{-}))))=\\=
\beta(\textrm{-},\textrm{-},f(\textrm{-},\textrm{-}))-H_1(\alpha(\textrm{-},g(\textrm{-},\textrm{-},\textrm{-})))=\beta(\textrm{-},\textrm{-},f(\textrm{-},\textrm{-}))-\gamma(\textrm{-},\textrm{-},\textrm{-},\textrm{-}),
\end{multline*}
\begin{multline*}
H_1(\alpha(g(\textrm{-},\textrm{-},\textrm{-}),\textrm{-}))=\beta(\textrm{-},f(\textrm{-},\textrm{-}),\textrm{-})+H_1(\alpha(g(\textrm{-},\textrm{-},\textrm{-}),\textrm{-})-D_1(\beta(\textrm{-},f(\textrm{-},\textrm{-}),\textrm{-})))=\\=
\beta(\textrm{-},f(\textrm{-},\textrm{-}),\textrm{-})+H_1(-\alpha(\textrm{-},-g(\textrm{-},\textrm{-},\textrm{-})))=\beta(\textrm{-},f(\textrm{-},\textrm{-}),\textrm{-})+\gamma(\textrm{-},\textrm{-},\textrm{-},\textrm{-}).
\end{multline*}
For $D_2$, we may use the formulas we already obtained, getting
\begin{multline*}
D_2(\omega(\textrm{-},\textrm{-},\textrm{-},\textrm{-}))=\beta(f(\textrm{-},\textrm{-}),\textrm{-},\textrm{-})-H_1D_1(\beta(f(\textrm{-},\textrm{-}),\textrm{-},\textrm{-}))=\\=
\beta(f(\textrm{-},\textrm{-}),\textrm{-},\textrm{-})-H_1(\alpha(f(f(\textrm{-},\textrm{-}),\textrm{-}),\textrm{-})+\alpha(f(\textrm{-},\textrm{-}),f(\textrm{-},\textrm{-})))=\\=
\beta(f(\textrm{-},\textrm{-}),\textrm{-},\textrm{-})-H_1(-\alpha(g(\textrm{-},\textrm{-},\textrm{-}),\textrm{-})+\alpha(f(\textrm{-},\textrm{-}),f(\textrm{-},\textrm{-})))=\\=
\beta(f(\textrm{-},\textrm{-}),\textrm{-},\textrm{-})+(\beta(\textrm{-},f(\textrm{-},\textrm{-}),\textrm{-})+\gamma(\textrm{-},\textrm{-},\textrm{-},\textrm{-}))-(\beta(\textrm{-},\textrm{-},f(\textrm{-},\textrm{-}))-\gamma(\textrm{-},\textrm{-},\textrm{-},\textrm{-}))=\\=
\beta(f(\textrm{-},\textrm{-}),\textrm{-},\textrm{-})+\beta(\textrm{-},f(\textrm{-},\textrm{-}),\textrm{-})-\beta(\textrm{-},\textrm{-},f(\textrm{-},\textrm{-}))+2\gamma(\textrm{-},\textrm{-},\textrm{-},\textrm{-}).
\end{multline*}
\end{proof}

In particular, when we use our resolution to compute $\Tor^{\widetilde{\As}}_\ldot(\k,\k)$, all summands killed by the 
augmentation vanish, and we get
 $$
d_1\beta=d_1\gamma=0, \quad d_2\omega=2\gamma,
 $$ 
so we see once again that $\Tor^{\widetilde{\As}}_2(\k,\k)$ is one-dimensional. (This result is not surprising: the second term of the bar homology encodes relations, and in our case the space of relations is one-dimensional, and $\beta$ is the leading term of that relation.)

Moreover, using results of Section~\ref{oper_minimal}, and computing the ``deformed'' differential, it is easy to check that in all arities less than~$8$ the homology is concentrated in one homological degree. Dimensions are summarised in the following

\begin{proposition}
We have
\begin{gather*}
\dim\Tor^{\widetilde{\As}}_1(\k,\k)(2)=1,\\
\dim\Tor^{\widetilde{\As}}_2(\k,\k)(3)=1,\\
\dim\Tor^{\widetilde{\As}}_3(\k,\k)(5)=4,\\
\dim\Tor^{\widetilde{\As}}_4(\k,\k)(6)=14,\\
\dim\Tor^{\widetilde{\As}}_5(\k,\k)(7)=30.
\end{gather*}
\end{proposition}

\section{Applications}\label{appl}

\subsection{Another proof of the PBW criterion for Koszulness}\label{PBW}

The goal of this section is to prove the following statement (which brings to the common ground the PBW criterion of Priddy~\cite{Priddy} for associative algebras and the PBW criterion of Hoffbeck~\cite{Hoffbeck} for operads).

\begin{theorem}\label{newPBW}
An associative algebra (commutative algebra, operad etc.) with a quadratic Gr\"obner basis is Koszul. 
\end{theorem}

\begin{proof}
First of all, it is enough to prove it in the monomial case, since it gives an upper bound on the homology: for the deformed differential, the cohomology may only decrease. In the monomial case, it follows from an easy observation that the free right module resolution we constructed actually coincides with the Koszul complex of the corresponding algebra (commutative algebra, operad etc.), and acyclicity of the Koszul complex is one of the equivalent criteria of Koszulness.
\end{proof}

\subsection{Operads and commutative algebras}\label{com-alg}

Recall a construction of an operad from a graded commutative algebra described in~\cite{Khor}.

Let $A$ be a graded commutative algebra. Define an operad $\calO_A$ as follows. We put $\calO_A(n):=A_{n-1}$, and let
the partial composition map 
 $$
\circ_i\colon\calO_A(k)\otimes\calO_A(l)=A_{k-1}\otimes A_{l-1}\to A_{k+l-2}=\calO_A(k+l-1)  
 $$
be the product in~$A$. 

As we remarked in \cite{DK}, a basis of the algebra~$A$ leads to a basis of the operad~$\calO_A$: product of generators of the polynomial algebra is replaced by the iterated composition of the corresponding generators of the free operad where each composition is substitution into the last slot of an operation. Assume that we know a Gr\"obner basis for the algebra~$A$ (as an associative algebra). It leads to a Gr\"obner basis for the operad~$\calO_A$ as follows: we first impose the quadratic relations defining the operad $\calO_{\k[x_1,\ldots,x_n]}$ coming from the polynomial algebra (stating that the result of a composition depends only on the operations composed, not on the order in which we compose operations), and then use the identification of relations in the polynomial algebra with elements of the corresponding operad, as above. Our next goal is to explain how to use the Anick resolution of the trivial module for~$A$ to construct a small resolution of the trivial module for~$\calO_A$. 

\begin{theorem}
There exists a free right module resolution 
 $$
\calR\circ\calO_A\to\k\to0
 $$ 
of the trivial~$\calO_A$-module. Here $\calR$ is the operad generated by all bar homology classes of $A$, where we put classes of internal degree~$k-1$ in arity~$k$, with relations $c_1\circ_k c_2=0$, for all classes $c_1,c_2$ where $c_1$ has arity~$k$. In other words, in $\calR$ all compositions are allowed except for those using the last slot of an operation.
\end{theorem}

\begin{proof}
This statement is almost immediate from our previous results. Indeed, we know how to obtain a Gr\"obner basis for $\calO_A$ from a Gr\"obner basis of $A$. If we apply the result of Theorem~\ref{min-resol-oper}, it is clear that the minimal resolution over the monomial replacement of~$\calO_A$ is of the same form as defined above, but we should start with the operad generated by chains, not by the homology. To obtain a resolution over $\calO_A$, let us look carefully into the general reconstruction scheme from the previous section. It recovers lower terms of differentials and homotopies by recalling lower terms of elements of the Gr\"obner basis. Let us do the reconstruction in two steps. At first, we shall recall all lower terms of relations except for those starting with $\alpha(\beta(\textrm{-},\textrm{-}),\textrm{-})$; the latter are still assumed to vanish. On the next step we shall recall all lower terms of those quadratic relations. Note that after the first step we model many copies of the associative algebra resolution and the differential there; so we can compute the homology explicitly. At the next step, a differential will be induced on this homology we computed, and we end up with a resolution of the required type. 
\end{proof}

In some cases the existence of such a resolution is enough to compute the bar homology of~$\calO_A$; for example, it is so when the algebra~$A$ is Koszul, as we shall see now. In general, the differential of this resolution incorporates lots of information, including the higher operations (Massey (co)products) on the homology of~$A$.

Recall that if the algebra $A$ is quadratic, then the operad $\calO_A$ is quadratic as well. In~\cite{DK}, we proved that if the algebra~$A$ is PBW, then the operad~$\calO_A$ is PBW as well, and hence is Koszul. Now we shall prove the following substantial generalisation of this statement (substantially simplifying the proof of this statement given in~\cite{Khor}).

\begin{theorem}\label{O_AisKoszul}
If the algebra $A$ is Koszul, then the operad $\calO_A$ is Koszul as well. 
\end{theorem}

\begin{proof}
Koszulness of our algebra implies that the homology of the bar resolution is concentrated on the diagonal. Consequently, the operad $\calR$ constructed above is automatically concentrated on the diagonal, and so is its homology,  which completes the proof.
\end{proof}

It is worth mentioning that the same can be applied to dioperads. For a commutative graded algebra~$A$, let us put $\calD_A(m,n):=A_{m+n-2}$, and let
the partial composition map 
\begin{multline*}
\circ_{i,j}\colon\calD_A(m,n)\otimes\calD_A(p,q)=A_{m+n-2}\otimes A_{p+q-2}\to\\ \to A_{m+n+p+q-4}=\calD_A(m+p-1, n+q-1)  
\end{multline*}
be the product in~$A$. The bi-collection $\{\calD_A(m,n)\}$ forms a dioperad, which is quadratic whenever the algebra~$A$ is, and, as it turns out, is Koszul whenever the algebra $A$ is. This can be proved similarly to how the previous theorem is proved.

\subsection{The operads of Rota--Baxter algebras}

The main goal of this section is to compute the bar homology for the operad of Rota--Baxter algebras, and the operad of noncommutative Rota--Baxter algebras. Those are among the simplest examples of operads which are not covered by the Koszul duality theory, being operads with nonhomogeneous relations.  

\subsubsection{The operads $RB$ and $ncRB$, and their Gr\"obner bases}

\begin{definition}
A commutative Rota--Baxter algebra of weight~$\lambda$ is a vector space with an associative commutative product $a,b\mapsto a\cdot b$ and a unary operator~$P$ which satisfy the following identity:
\begin{equation}\label{RBeq}
P(a)\cdot P(b)=P(P(a)\cdot b+a\cdot P(b)+\lambda a\cdot b). 
\end{equation}
We denote by $RB$ the operad of Rota--Baxter algebras. We view it as a shuffle operad with one binary and one unary generator.   
\end{definition}

Commutative Rota--Baxter algebras were defined by Rota~\cite{Rota} who was inspired by work of Baxter \cite{Baxter} in probability theory. Various constructions of free commutative Rota--Baxter algebras were given by Rota, Cartier \cite{Cartier} and, more recently, Guo and Keigher \cite{GKe}. The latter paper also contains extensive bibliography and information on various applications of those algebras.

\begin{definition}
A noncommutative Rota--Baxter algebra of weight~$\lambda$ is a vector space with an associative product $a,b\mapsto a\cdot b$ and a unary operator~$P$ which satisfy the same identity as above:
 $$
P(a)\cdot P(b)=P(P(a)\cdot b+a\cdot P(b)+\lambda a\cdot b). 
 $$
We denote by $ncRB$ the operad of noncommutative Rota--Baxter algebras. Somehow, it is a bit simpler than the operad in the commutative case, because it can be viewed as a nonsymmetric operad with one binary and one unary generator.   
\end{definition}

Noncommutative Rota--Baxter algebras has been extensively studied in the past years. We refer the reader to the paper of Ebrahimi--Fard and Guo \cite{EF-G2} for an extensive discussion of applications and occurrences of those algebras in various areas of mathematics, and a combinatorial construction of the corresponding free algebras.

Let us consider the path-lexicographic ordering of the free operad; we assume that $P>\cdot$. 

\begin{proposition}
The defining relations for operads $RB$ and $ncRB$ form a Gr\"obner basis.  
\end{proposition}

\begin{proof}
Here we present a proof for the case of $ncRB$, the proof for $RB$ is essentially the same, with the only exception that there are two S-polynomials to be reduced, as opposed to one S-polynomial in the case of $ncRB$ (which, as we pointed above, is easier because we are dealing with a nonsymmetric operad).
 
For the associative suboperad of $ncRB$, the defining relations form a Gr\"obner basis, so the S-polynomials coming from the small common multiples the leading term of the associativity relation has with itself clearly can be reduced to zero. The leading term of the Rota--Baxter relation is $P(P(a_1)a_2)$. This term only has a nontrivial overlap with itself, not with the leading term of the associativity relation, and that overlap is $P(P(P(a_1)\cdot a_2)\cdot a_3)$. From this overlap, we compute the S-polynomial 
\begin{multline*}
-P(P(a_1\cdot P(a_2))\cdot a_3)-\lambda P(P(a_1\cdot a_2)\cdot a_3)+P((P(a_1)\cdot P(a_2))\cdot a_3)+\\
+P((P(a_1)\cdot a_2)\cdot P(a_3))+\lambda P((P(a_1)\cdot a_2)\cdot a_3)-P(P(a_1)\cdot a_2)\cdot P(a_3), 
\end{multline*}
and it can be reduced to zero by a lengthy sequence of reductions which we omit here (but which in fact can be read from the formula for $d\nu_3$ in Proposition~\ref{differential} below). By Diamond Lemma~\cite{DK}, our relations form a Gr\"obner basis.  
\end{proof}

In the case of the operad $ncRB$, our computation immediately provides bases for free noncommutative Rota--Baxter algebras. Indeed, since our operad is nonsymmetric, the degree $n$ component of the free noncommutative Rota--Baxter algebra generated by the set $B$ is nothing but $ncRB(n)\otimes V^{\otimes n}$, where $V=\mathop{\mathrm{span}}(B)$, so we can use the above Gr\"obner basis to describe that component. More precisely, we first define the set of admissible expressions on a set $B$ recursively as follows:
\begin{itemize}
 \item elements of $B$ are admissible expressions;
 \item if $b$ is an admissible expression, then $P(b)$ is an admissible expression;
 \item if $b_1,\ldots,b_k$ are admissible expressions, and for each $i$ either $b_i$ is an element of $B$ or $b_i=P(b_i')$ with $b_i'$ an admissible expression, then their associative product $b_1\cdot b_2\cdot\ldots\cdot b_k$ is an admissible expression.
\end{itemize}
Based on this definition, we shall call some of admissible expressions the Rota--Baxter monomials, tracing the construction of an admissible expression and putting some restrictions. Namely,
\begin{itemize}
 \item elements of $B$ are Rota--Baxter monomials;
 \item if $b$ is a Rota--Baxter monomial, which, as an admissible expression, is either $b=P(b')$ or $b=b_1\cdot b_2\cdot\ldots\cdot b_k$ with $b_1\in B$, then $P(b)$ is a Rota--Baxter monomial;
 \item if $b_1,\ldots,b_k$ are Rota--Baxter monomials, and for each $i$ either $b_i\in B$ or $b_i=P(b_i')$ for some $b_i'$, then their associative product $b_1\cdot b_2\cdot\ldots\cdot b_k$ is a Rota--Baxter monomial.
\end{itemize}

Our previous discussion means that the following result holds:
\begin{theorem}
The set of all Rota--Baxter monomials forms a basis in the free noncommutative Rota--Baxter algebra generated by the set~$B$. 
\end{theorem}

It would be interesting to compare this basis with the basis from~\cite{EF-G1}.

\subsubsection{Bar homology}

In this section, we compute the bar homology for both operads $RB$ and $ncRB$.

\begin{proposition}
For each of the operads $RB$ and $ncRB$, the resolution for its monomial version from Sec.4.2 is minimal, that is the differential induced on the space of generators is zero. 
\end{proposition}

\begin{proof}
In the case of the operad $RB$, the overlaps obtained from the leading monomials $(a_1\cdot a_2)\cdot a_3$, $(a_1\cdot a_3)\cdot a_2$, and $P(P(a_1)\cdot a_2)$ are, in arity $n$,
 $$
((\ldots((a_1\cdot a_{i_2})\cdot a_{i_3})\cdot \ldots )\cdot a_{i_n} \quad\text{and}\quad P(P(P(\ldots P(P(P(a_1)\cdot a_{i_2})\cdot a_{i_3})\ldots )\cdot a_{i_{n-1}})\cdot a_{i_n}),
 $$
for all permutations $i_2,i_3\ldots,i_n$ of integers $2,3,\ldots,n$. It is easy to see that for each of them there exists only one indecomposable covering by relations, so the differential maps such a generator to the space of decomposable elements, and the statement follows. 

Similarly, in the case of the operad $ncRB$, the only overlaps obtained from the leading monomials $(a_1\cdot a_2)\cdot a_3$ and $P(P(a_1)\cdot a_2)$ are, in arity $n$,
 $$
((\ldots((a_1\cdot a_2)\cdot a_3)\ldots )\cdot a_{n-1})\cdot a_n \quad\text{and}\quad P(P(P(\ldots P(P(P(a_1)\cdot a_2)\cdot a_3)\ldots )\cdot a_{n-1})\cdot a_n).
 $$
It is easy to see that for each of them there exists only one indecomposable covering by relations, so the differential maps such a generator to the space of decomposable elements, and the statement follows. 
\end{proof}

\begin{theorem}
We have 
\begin{itemize}
 \item  $$
\dim H^l(\mathbf{B}(RB))(k)=
\begin{cases}
(k-1)!, l=k\ge1,\\  
(k-1)!, l=k+1\ge2.\\   
\end{cases}
 $$   
\item $$
\dim H^l(\mathbf{B}(ncRB))(k)=
\begin{cases}
1, l=k\ge1,\\   
1, l=k+1\ge2.  
\end{cases}
 $$   
\end{itemize}
\end{theorem}

\begin{proof}
In both cases, the subspace of generators of the free resolution splits into two parts: the part obtained as overlaps of the leading terms of the associativity relations, and the part obtained as overlaps of the leading term of the Rota--Baxter relation with itself. In arity $k$, the former are all of homological degree $k-1$, while the latter --- of homological degree~$k$. This means that when we compute the homology of the differential of our resolution restricted to the space of generators, the only cancellations can happen if some of the elements resolving the associativity relation appear as differentials of some elements resolving the Rota--Baxter relation. However, it is clear that all the terms appearing in the formulas for the latter are of degree at least~$1$ in~$P$, so no cancellations are impossible.
\end{proof}

In addition to the bar homology computation, one can ask for explicit formulas for differentials in the free resolutions. It is not difficult to write down formulas for small arities (see the example below), but in general compact formulas are yet to be found. We expect that they incorporate the Spitzer's identity and its noncommutative analogue \cite{EF-GB-P}. However, the following statement is easy to check.

\begin{corollary}
\begin{itemize}
 \item The minimal model $RB_\infty$ for the operad $RB$ is a quasi-free operad whose space of generators has a $(k-1)!$-dimensional space of generators of homological degree~$(k-2)$ in each arity~$k\ge2$, and a $(k-1)!$-dimensional space of generators of homological degree~$k-1$ in each arity~$k\ge1$.
 \item The minimal model $ncRB_\infty$ for the operad $ncRB$ is a quasi-free operad generated by operations $\mu_k (k\ge2)$ of homological degree~$k-2$, and $\nu_l (l\ge1)$ of homological degree~$l-1$. 
\end{itemize}
\end{corollary}

Let us conclude this section with formulas for low arities differentials in $ncRB_\infty$, to give the reader a flavour of what sort of formulas to expect.       

\begin{example}\label{differential}
We have
 $$
d\mu_2=0,
 $$
 $$
d\mu_3=\mu_2(\mu_2(\textrm{-},\textrm{-}),\textrm{-})-\mu_2(\textrm{-},\mu_2(\textrm{-},\textrm{-})),
 $$
 $$
d\nu_1=0,
 $$
 $$
d\nu_2=P(\mu_2(P(\textrm{-}),\textrm{-}))+P(\mu_2(\textrm{-},P(\textrm{-})))-\mu_2(P(\textrm{-}),P(\textrm{-}))+\lambda P(\mu_2(\textrm{-},\textrm{-})),
 $$ 
\begin{multline*}
d\nu_3=
\nu_2(\mu_2(P(\textrm{-}),\textrm{-}),\textrm{-})-\nu_2(\textrm{-},\mu_2(P(\textrm{-}),\textrm{-}))+\nu_2(\mu_2(\textrm{-},P(\textrm{-})),\textrm{-})-\nu_2(\textrm{-},\mu_2(\textrm{-},P(\textrm{-})))+\\ 
+P(\mu_2(\nu_2(\textrm{-},\textrm{-}),\textrm{-})+\mu_2(\textrm{-},\nu_2(\textrm{-},\textrm{-})))+\mu_2(\nu_2(\textrm{-},\textrm{-}),P(\textrm{-}))-\mu_2(P(\textrm{-}),\nu_2(\textrm{-},\textrm{-}))+\\
+\mu_3(P(\textrm{-}),P(\textrm{-}),P(\textrm{-}))+\mu_3(P(\textrm{-}),P(\textrm{-}),\textrm{-})-P(\mu_3(P(\textrm{-}),\textrm{-},P(\textrm{-}))+\mu_3(\textrm{-},P(\textrm{-}),P(\textrm{-})))+\\
+\lambda\left[\nu_2(\mu_2(\textrm{-},\textrm{-}),\textrm{-})-\nu_2(\textrm{-},\mu_2(\textrm{-},\textrm{-}))-\right.\\ \left.-P(\mu_3(P(\textrm{-}),\textrm{-},\textrm{-})+\mu_3(\textrm{-},P(\textrm{-}),\textrm{-})+\mu_3(\textrm{-},\textrm{-},P(\textrm{-})))\right]-\\-\lambda^2P(\mu_3(\textrm{-},\textrm{-},\textrm{-}))
\end{multline*}
\end{example}

\subsection{The operad $BV$ and hypercommutative algebras}\label{BV}

The main goal of this section is to explain how our results can be used to study the operad~$BV$ of Batalin--Vilkovisky algebras. The key result below (Theorem~\ref{BV=Grav+delta}) is among those announced by Drummond-Cole and Vallette earlier \cite{DCVOberwolfach,ValBV} (see also~\cite{DCV}); our proofs are based on methods entirely different from theirs. 

\subsubsection{The operad $BV$ and its Gr\"obner basis.}

Batalin--Vilkovisky algebras show up in various questions of mathematical physics. In \cite{GTV}, a cofibrant resolution for the corresponding operad was presented. However, that resolution is a little bit more that minimal. In this section, we present a minimal resolution for this operad in the shuffle category. The operad~$BV$, as defined in most sources, is an operad with quadratic--linear relations: the odd Lie bracket can be expressed in terms of the product and the unary operator. However, alternatively one can say that a $BV$-algebra is a dg-commutative algebra with a unary operator $\Delta$ which is a differential operator of order at most~$2$. This definition of a $BV$-algebra  is certainly not new, see, e.~g., \cite{Getzler}. With this presentation, the corresponding operad becomes an operad with homogeneous relations (of degrees~$2$ and~$3$). Our choice of degrees and signs is taken from~\cite{GTV} where it is explained how to translate between this convention and other popular definitions of $BV$-algebras.

\begin{definition}[Batalin--Vilkovisky algebras with homogeneous relations]
A \emph{Batalin-Vilkovisky algebra}, or \emph{$BV$-algebra} for short,  is a differential graded vector space $(A, d_A)$ endowed with
\begin{itemize}
\item[-] a symmetric binary product $\bullet$ of degree~$0$,
\item[-]  a unary operator $\Delta$ of degree~$+1$,
\end{itemize}
such that $(A,d_A,\Delta)$ is a bicomplex, $d_A$ is a derivation with respect to the product, and such that
\begin{itemize}
\item[-] the product $\bullet$ is associative,
\item[-] the operator $\Delta$ satisfies $\Delta^2=0$,
\item[-] the operations satisfy the cubic identity 
\begin{equation}\label{BV3}
 \Delta(\,\textrm{-}\,\bullet\,\textrm{-}\,\bullet\,\textrm{-}\,) =
((\Delta(\,\textrm{-}\,\bullet\,\textrm{-})\,\bullet\,\textrm{-}\,)-(\Delta(\,\textrm{-}\,)\,\bullet\,\textrm{-}\,\bullet\,\textrm{-}\,)).(1+(123)+(132)),
\end{equation}
\end{itemize}
\end{definition}

In what follows, it is very helpful to have in mind the computations of Section~\ref{like-BV}: phenomena discovered there for an associative algebra $\k\langle x,y\mid y^2, x^2y+xyx+yx^2\rangle$ are very similar to the phenomena we shall observe for the operad $BV$. Informally, one should think of the generator $y$ of the abovementioned algebra as of an analogue of the Batalin--Vilkovisky operator $\Delta$, and of the generator $x$ as of an analogue of the binary product $\textrm{-}\bullet\textrm{-}$.

Let us consider the ordering of the free operad where we first compare lexicographically the operations on the paths from the root to leaves, and then the planar permutations of leaves; we assume that $\Delta>\bullet$. 

\begin{proposition}
The above relations together with the degree~$4$ relation
\begin{equation}\label{BV4}
(\Delta(\,\textrm{-}\,\bullet\,\Delta(\,\textrm{-}\,\bullet\,\textrm{-}))-
\Delta(\Delta(\,\textrm{-}\,)\,\bullet\,\textrm{-}\,\bullet\,\textrm{-}\,)).(1+(123)+(132))=0
\end{equation}
form a Gr\"obner basis of relations for the operad of BV-algebras.
\end{proposition}

\begin{proof}
Here and below we use the language of operations, as opposed the language of tree monomials; our operations reflect the structure of the corresponding tree monomials in the free shuffle operad. For each~$i$, the argument~$a_i$ of an operation corresponds to the leaf~$i$ of the corresponding tree monomial. 

With respect to our ordering, the leading monomials of our original relations are $(a_1\bullet a_2)\bullet a_3$, $(a_1\bullet a_3)\bullet a_2$, $\Delta^2(a_1)$, and $\Delta(a_1\bullet(a_2\bullet a_3))$. The only small common multiple of $\Delta^2(a_1)$ and $\Delta(a_1\bullet(a_2\bullet a_3))$ gives a nontrivial S-polynomial which, is precisely the relation~\eqref{BV4}. The leading term of that relation is $\Delta(\Delta(a_1\bullet a_2) \bullet a_3)$. 

It is well known that $\dim BV(n)=2^nn!$ \cite{Getzler}, so to verify that our relations form a Gr\"obner basis, it is sufficient to show that the restrictions imposed by these leading monomials are strong enough, that is that the number of arity~$n$ tree monomials that are not divisible by any of these is equal to $2^nn!$. Moreover it is sufficient to check that for $n\le 4$, since all S-polynomials of our relations will be elements of arity at most~$4$. This can be easily checked by hand, or by a computer program~\cite{DVJ}.
\end{proof}

\subsubsection{Bar homology of the operad~$BV$.}

Let us denote by~$\calG$ the Gr\"obner basis from the previous section.

\begin{proposition}
For the monomial version of $BV$, the resolution~$\calA_\calG$ from Section~\ref{oper_monom} is minimal, that is the differential induced on the space of generators is zero. 
\end{proposition}

\begin{proof}
Let us describe explicitly the space of generators, that is possible indecomposable coverings of monomials by leading terms of relations (all monomials below are chosen from the basis of the free shuffle operad, so the correct ordering of subtrees is assumed).
These are
\begin{itemize}
 \item[-] all monomials $\Delta^{k}(a_1)$, $k\ge2$ (covered by several copies of $\Delta^2(a_1)$),
 \item[-] all ``Lie monomials''
\begin{equation}\label{Lie}
\lambda=(\ldots((a_1\bullet a_{k_2})\bullet a_{k_3})\bullet \ldots )\bullet a_{k_n}
\end{equation}
where $(k_2,\ldots,k_n)$ is a permutation of numbers $2$, \ldots, $n$ (only the leading terms $(a_1\bullet a_2)\bullet a_3$ and $(a_1\bullet a_3)\bullet a_2$ are used in the covering),
 \item[-] all the monomials
\begin{equation}\label{LT3}
\Delta^k(\Delta(\lambda_1\bullet(\lambda_2\bullet\lambda_3)))
\end{equation}
where $k\ge1$, each~$\lambda_i$ is a Lie monomial as described above (several copies of $\Delta^2$, the leading term of degree~$3$, and several Lie monomials are used), 
 \item[-] all monomials
\begin{equation}\label{LT4}
\Delta^k(\Delta(\ldots\Delta(\Delta(\lambda_1\bullet \lambda_{2})\bullet \lambda_{3})\bullet \ldots )\bullet \lambda_{n}) 
\end{equation}
where $k\ge 0$, $n\ge 3$, and $\lambda_i$ are Lie monomials (several copies of all leading terms are used, including at least one copy of the degree~$4$ leading term).
\end{itemize}
This is a complete list of tree monomials $T$ for which $\calA_\calG^T$ is nonzero in positive homological degrees. It is easy to see that for each of them there exists only one indecomposable covering by relations, that is only one generator of $\calA_\calG$ of shape $T$. Consequently, the differential maps such a generator to $(\calA_\calG)_+^2$, so the differential induced on generators is identically zero.
\end{proof}

The resolution of the operad~$BV$ which one can derive by our methods from this one is quite small (in particular, smaller than the one of~\cite{GTV}) but still not minimal. However, we now have enough information to compute the bar homology of the operad~$BV$.

\begin{theorem}\label{BVinf}
The basis of $H(\mathbf{B}(BV))$ is formed by monomials 
 $$
\Delta^k(a_1), \quad k\ge 1,
 $$
and all monomials of the form
\begin{equation}\label{BVbasis}
\underbrace{\Delta(\ldots\Delta(\Delta(}_{n-1\textrm{ times }}\lambda_1\bullet \lambda_{2})\bullet \ldots )\bullet (\lambda_{n}\bullet a_j)), \quad n\ge 1
\end{equation}
from the monomial resolution discussed above. Here all $\lambda_i$ are Lie monomials.
\end{theorem}

\begin{proof}
First of all, let us notice that since $\Omega(\mathbf{B}(BV))$, a free operad generated by $\mathbf{B}(BV)[-1]$, provides a resolution for $BV$, the space $H(\mathbf{B}(BV))[-1]$ is the space of generators of the minimal free resolution, and we shall study the resolution provided by our methods.

Similarly to how things work for the operad~$\widetilde{\As}$ in Section~\ref{AntiAsExample}, it is easy to check that the element $\Delta(\Delta(a_1\bullet a_2) \bullet a_3)$ that corresponds to the leading term of the only contributing S-polynomial will be killed by the differential of the element $\Delta^2(a_1\bullet(a_2\bullet a_3))$ (covered by two leading terms $\Delta^2(a_1)$ and $\Delta(a_1\bullet(a_2\bullet a_3))$)  in the deformed resolution. This observation goes much further, namely
we have for $k\ge1$
\begin{multline}
D(\Delta^k(\Delta(\ldots\Delta(\Delta(\lambda_1\bullet \lambda_{2})\bullet \lambda_{3})\bullet \ldots )\bullet (\lambda_n\bullet a_{j}))=\\ 
=\Delta^{k-1}((\Delta(\ldots\Delta(\Delta(\lambda_1\bullet \lambda_{2})\bullet \lambda_{3})\bullet \ldots )\bullet \lambda_{n})\bullet a_{j})+\text{lower terms}
\end{multline}
in the sense of the partial ordering we discussed earlier). So, if we retain only leading terms of the differential, the resulting homology classes are represented by all the monomials of arity~$m$
\begin{equation}
\Delta(\ldots\Delta(\Delta(\lambda_1\bullet \lambda_{2})\bullet \ldots )\bullet \lambda_{n}) 
\end{equation}
with $\lambda_n$ having at least two leaves. They all have the same homological degree~$m-2$ in the resolution, and so there are no further cancellations. 
\end{proof}

So far we have not been able to describe a minimal resolution of the operad~$BV$ by relatively compact closed formulas, even though in principle our proof, once processed by a version of Brown's machinery~\cite{Brown,Cohen}, would clearly yield such a resolution (in the shuffle category). 

\subsubsection{Operads $\Hycom$ and $\Grav$.}
The operads $\Hycom$ and its Koszul dual $\Grav$ were originally defined in terms of moduli spaces of curves of genus~$0$ with marked points $\calM_{0,n+1}$ \cite{Getzler1,GK}. However, we are interested in the algebraic aspects of the story, and we use the following descriptions of these operads as quadratic algebraic operads~\cite{Getzler1}. An algebra over~$\Hycom$ is a chain complex $A$ with a sequence of
graded symmetric products 
 $$(x_1,\dots,x_n)\colon A^{\otimes n}\to A$$ 
of degree
$2(n-2)$, which satisfy the following relations (here $a,b,c,x_1,\dots,x_n$, $n\ge0$, are elements of $A$):
\begin{equation}\label{hycom_rel}
\sum_{S_1\amalg S_2=\{1,\dots,n\}} \pm ((a,b,x_{S_1}),c,x_{S_2})
= \sum_{S_1\amalg S_2=\{1,\dots,n\}} \pm (a,(b,c,x_{S_1}),x_{S_2}).
\end{equation}
Here, for a finite set $S=\{s_1,\dots,s_k\}$, $x_S$ denotes for $x_{s_1},\dots,x_{s_k}$, and $\pm$ means the Koszul
sign rule.

An algebra over~$\Grav$ is a chain complex with graded antisymmetric products
$$[x_1,\dots,x_n]\colon A^{\otimes n}\to A$$ of degree $2-n$, which satisfy the
relations: 
\begin{multline}\label{relgrav}
\sum_{1\le i<j\le k}
\pm [[a_i,a_j],a_1,\dots,\widehat{a_i},\dots,\widehat{a_j},\dots,a_k,
b_1,\dots,b_\ell]
=\\= \begin{cases} [[a_1,\dots,a_k],b_1,\dots,b_l] , & l>0 , \\
0 , & l=0, \end{cases} 
\end{multline}
for all $k>2$, $l\ge0$, and $a_1,\dots,a_k,b_1,\dots,b_l\in A$. For example, setting $k=3$ and $l=0$, we obtain the Jacobi relation for~$[a,b]$. (Similarly, the first relation for~$\Hycom$ is the associativity of the product~$(a,b)$.)

Let us define an admissible ordering of the free operad whose quotient is $\Grav$ as follows. We introduce an additional weight grading, putting the weight of the corolla corresponding to the binary bracket equal to~$0$, all other weights of corollas equal to~$1$, and extending it to compositions by additivity of weight. To compare two monomials, we first compare their weights, then the root corollas, and then path sequences~\cite{DK} according to the reverse path-lexicographic order. For both of the latter steps, we need an ordering of corollas; we assume that corollas of larger arity are smaller. Then for the relation $(k,l)$ in \eqref{relgrav} (written in the shuffle notation with variables in the proper order), its leading monomial is equal to the monomial in the right hand side for $l>0$, and to the monomial $[a_1,\ldots,a_{n-2},[a_{n-1},a_n]]$ for $l=0$.

The following theorem, together with the PBW criterion, implies that the operads $\Grav$ and $\Hycom$ are Koszul, the fact first proved by Getzler~\cite{Getzler}.

\begin{theorem}\label{Grav}
For our ordering, the relations of $\Grav$ form a Gr\"obner basis of relations. 
\end{theorem}

\begin{proof}
The tree monomials that are not divisible by leading terms of relations are precisely
\begin{equation}\label{ltGrav}
[\lambda_1,\lambda_2,\ldots,\lambda_{n-1},a_j],  
\end{equation}
where all $\lambda_i$, $1\le i\le(n-1)$ are Lie monomials as in~\eqref{Lie} (but made from brackets, not products).
\begin{lemma}\label{grav}
The graded character of the space of such elements of arity~$n$ is 
\begin{equation}\label{gravchar}
(2+t^{-1})(3+t^{-1})\ldots(n-1+t^{-1}).  
\end{equation}
\end{lemma}
\begin{proof}
To compute the number of basis elements where the top degree corolla is of arity $k+1$ (or, equivalently, degree~$1-k$), $k\ge1$, let us notice that this number is equal to the number of basis elements
 $$
[\lambda_1,\lambda_2,\ldots,\lambda_k] 
 $$
where the arity of $\lambda_k$ is at least~$2$ (a simple bijection: join $\lambda_{n-1}$ and $a_j$ into $[\lambda_{n-1},a_j]$). The latter number is equal to
\begin{equation}\label{sum_binom}
\sum_{\substack{m_1+\ldots+m_k=n, \\m_i\ge1, m_k\ge2}}\frac{(m_1-1)!(m_2-1)!\ldots(m_k-1)!m_1m_2\cdot\ldots\cdot m_k}{(m_1+m_2+\ldots+m_k)(m_2+\ldots+m_k)\cdot\ldots\cdot m_k}\binom{m_1+\ldots+m_k}{m_1,m_2,\ldots,m_k}  
\end{equation}
where each factor $(m_i-1)!$ counts the number of Lie monomials of arity~$m_i$, and the remaining factor is the number of shuffle permutations of the type $(m_1,\ldots,m_k)$ (\cite{DVJ}). This can be rewritten in the form 
 $$
\sum_{m_1+\ldots+m_k=n, m_i\ge1, m_k\ge2}\frac{(m_1+\ldots+m_k-1)!}{(m_2+\ldots+m_k)(m_3+\ldots+m_k)\cdot\ldots\cdot m_k}
 $$
and if we introduce new variables $p_i=m_i+\ldots+m_k$, it takes the form
 $$
\sum_{2\le p_{k-1}<\ldots<p_1\le n-1}\frac{(n-1)!}{p_2\ldots p_k},
 $$
which clearly is the coefficient of $t^{1-k}$ in the product
\begin{multline}
(n-1)!\left(1+\frac{1}{2t}\right)\left(1+\frac{1}{3t}\right)\cdot\ldots\cdot\left(1+\frac{1}{(n-1)t}\right)=\\=\left(2+t^{-1}\right)\left(3+t^{-1}\right)\ldots\left(n-1+t^{-1}\right). 
\end{multline}
\end{proof}
Since the graded character of $\Grav$ is given by the same formula~\cite{Getzler1}, we indeed see that the leading terms of defining relations give an upper bound on dimensions homogeneous components of $\Grav$ that coincides with the actual dimensions, so there is no room for further Gr\"obner basis elements. 
\end{proof}

\subsubsection{$BV_\infty$ and hypercommutative algebras}

\begin{theorem}\label{BV=Grav+delta}
On the level of collections of graded vector spaces, we have
\begin{equation}\label{BVinfty}
H(\mathbf{B}(BV))[-1]\simeq\Grav^*\otimes\End_{\k[1]}\oplus \delta\k[\delta], 
\end{equation}
where $\Grav^*$ is the cooperad dual to $\Grav$, and $\delta\k[\delta]$ is a cofree coalgebra generated by an element $\delta$ of degree~$2$.
\end{theorem}

\begin{proof}
As above, instead of looking at the bar complex, we shall study the basis of the space of generators of the minimal resolution obtained in Theorem~\ref{BVinf}. In arity~$1$, the element $\delta^k$ (of degree $2k$) corresponds to $\Delta^k(a_1)[-1]$ (of degree $k+(k-1)+1=2k$, the first summand coming from the fact that $\Delta$ is of degree~$1$, the second from the fact that $\Delta^k$ is an overlap of $k-1$ relations, and the last one is the degree shift). The case of elements of internal degree~$0$ (which in both cases are Lie monomials) is also obvious; a Lie monomial of arity $n$ in the space of generators of the free resolution is of homological degree~$n-2+1=n-1$, the second summand coming from the degree shift, and this matches the degree shift given by~$\End_{\k[1]}(n)$. For elements of internal degree $k-1$, let us extract from a typical monomial 
 $$
\underbrace{\Delta(\ldots\Delta(\Delta(}_{k-1\textrm{ times }}\lambda_1\bullet \lambda_{2})\bullet \ldots )\bullet (\lambda_{k}\bullet a_j)), 
 $$
of this degree and of arity $n$ the Lie monomials $\lambda_1, \lambda_2, \ldots, \lambda_{k-1}, \lambda_{k}, a_j$, and assign to this the element of $\Grav^*\otimes\End_{\k[1]}$ corresponding to the dual element of the monomial $[\lambda_1,\lambda_2,\ldots,\lambda_{k-1},\lambda_k,a_j]$ in the gravity operad. This establishes a degree-preserving bijection, because if arities of $\lambda_1,\ldots,\lambda_k$ are $n_1,\ldots,n_k$, the total (internal plus homological) degree of the former element is $(k-1)+(k-2+1+(n_1-1)+\ldots+(n_k-1))+1=n+k-2$ (where we add up the $\Delta$ degree, the overlap degree, and the degree shift by~$1$), and the total degree of the latter one is $(k-1)+(n-1)=n+k-2$.
\end{proof}

We conclude this paragraph with a discussion on how our results match those of Barannikov and Kontsevich (\cite{BK}, see also~\cite{LS,Manin}) who proved in a rather indirect way that for a dg$BV$-algebra that satisfies the ``$\partial-\overline{\partial}$-lemma'', there exists a $\Hycom$-algebra structure on its cohomology. Their result hints that our isomorphism~\eqref{BVinfty} exists not just on the level of graded vector spaces, but rather has some deep operadic structure behind it. For precise statements and more details we refer the reader to~\cite{DCV}.

From Theorem~\ref{Grav}, it follows that the operads $\Grav$ and $\Hycom$ are Koszul, so $\Omega(\Grav^*\otimes\End_{\k[1]})$ is a minimal model for $\Hycom$. More precisely, we shall show that the differential of $BV_\infty$ on generators coming from $\Grav^*$ deforms the differential of~$\Hycom_\infty$ in the following sense. Let $D$ and $d$ denote the differentials of $BV_\infty$ and $\Hycom_\infty$ respectively. We can decompose $D=D_2+D_3+\ldots$ (respectively $d=d_2+d_3+\ldots$) according to the $\infty$-cooperad structure it provides on the space of generators. Also, let $m^*$ denote the obvious coalgebra structure on $\delta\k[\delta]$. We shall call a tree monomial in $BV_\infty$ \emph{mixed}, if it contains both corollas from $\Grav^*\otimes\End_{\k[1]}$ and from $(\delta\k[\delta])$. Then we have 
\begin{equation}\label{d2}
D_2=d_2+m^*, 
\end{equation}
while for $k\ge3$ the co-operation $D_k$ is zero on the generators $\delta\k[\delta]$, and maps generators from $\Grav^*$ into linear combinations of mixed tree monomials. Indeed, the result of Barannikov and Kontsevich \cite{BK} essentially implies that there exists a mapping from $\Hycom$ to the homotopy quotient $BV/\Delta$. In fact, it is an isomorphism, which can be proved in several different ways, both using Gr\"obner bases and geometrically; see \cite{Markarian} for a short geometric argument proving that. This means that the following maps exist (the vertical arrows are quasiisomorphisms between the operads and their minimal models):
 $$
\xymatrix{
BV_\infty \ar@{->>}[d] \ar@{->>}[dr]_{\pi} & & \Hycom_\infty \ar@{->>}[d]  \\ 
BV    &BV/\Delta   & \Hycom \ar@{->}[l]_{\widetilde{\hphantom{aaaaa}}}
}
 $$
Lifting $\pi\colon BV_\infty\to BV/\Delta\simeq\Hycom$ to the minimal model $\Hycom_\infty$ of $\Hycom$, we obtain the commutative diagram
 $$
\xymatrix{
BV_\infty \ar@{->>}[rr]^{\psi} \ar@{->>}[d] \ar@{->>}[dr]_{\pi}& & \Hycom_\infty \ar@{->>}[d]  \\ 
BV    &BV/\Delta  & \Hycom\ar@{->}[l]_{\widetilde{\hphantom{aaaaa}}}  
}
 $$
so there exists a map of dg-operads (and not just graded vector spaces, as it follows from our previous computations) between $BV_\infty$ and $\Hycom_\infty$. Commutativity of our diagram together with simple degree considerations yields what we need.

\bibliographystyle{amsplain}
\bibliography{monomial-new}

\end{document}